\documentclass{article}
\usepackage[margin=1.5in]{geometry}
\usepackage[utf8]{inputenc}
\usepackage{graphicx}
\usepackage{enumerate}
\usepackage{hyperref}

\newcommand{\email}[1]{\protect\href{mailto:#1}{#1}}
\usepackage[numbers,sort&compress]{natbib}
\usepackage{bbm}
\usepackage{amsmath,amssymb,amsfonts,amsthm}

\newtheorem{theorem}{Theorem}[section]

\newtheorem{proposition}[theorem]{Proposition}
\newtheorem{lemma}[theorem]{Lemma}

\theoremstyle{definition}
\newtheorem{definition}[theorem]{Definition}

\newtheorem{example}[theorem]{Example}
\newtheorem{remark}[theorem]{Remark}

\usepackage{mathtools}
\usepackage{nicefrac}
\usepackage{xcolor}

\newcommand{\N}{\mathbb{N}}
\newcommand{\Z}{\mathbb{Z}}
\newcommand{\R}{\mathbb{R}}
\newcommand{\C}{\mathbb{C}}

\newcommand{\T}{\mathbb{T}}
\newcommand{\di}{\,\mathrm{d}}

\newcommand{\expon}[1]{\mathrm{e}^{#1}}
\newcommand{\imag}{\mathrm{i}}
\newcommand{\smallnorm}[1]{\lVert#1\rVert}

\newcommand{\norm}[1]{\left\lVert#1\right\rVert}
\newcommand{\abs}[1]{\left\lvert#1\right\rvert}
\newcommand{\enorm}[1]{\abs{#1}}

\newcommand{\domainsymb}{\mathcal{D}}
\newcommand{\rangesymb}{\mathcal{R}}

\newcommand{\domain}[1]{\domainsymb({#1})}
\newcommand{\range}[1]{\rangesymb({#1})}

\newcommand{\restrict}[2]{{\left.\kern-\nulldelimiterspace #1\right|_{#2}}}

\newcommand{\characteristic}[1]{\mathbbm{1}_{#1}}

\newcommand{\ball}{\mathcal{B}}
\newcommand{\ballorigin}[1]{\ball\left({#1}\right)}

\newcommand{\fourier}{\mathcal{F}}
\newcommand{\fhat}[1]{\widehat{#1}}

\newcommand{\pseudoinv}[1]{{#1}^\dagger}

\DeclareMathOperator*{\argmin}{arg\,min}

\title{On $L^\infty$ stability for wave propagation \\ and for  linear inverse problems}
\author{
Rima Alaifari\thanks{
                Department of Mathematics, RWTH Aachen University
(\email{rima.alaifari@rwth-aachen.de}). Work done while at Seminar for Applied Mathematics, Department of Mathematics, ETH Zurich.
                }
\and Giovanni S.\ Alberti\thanks{
                Machine Learning Genoa Center (MaLGa), Department of Mathematics, Department of Excellence 2023–2027, University of Genoa (\email{giovanni.alberti@unige.it}).
                }
\and Tandri Gauksson\thanks{Seminar for Applied Mathematics, Department of Mathematics, ETH Zurich
(\email{tandri.gauksson@decode.is}). Current address: deCODE genetics,  Reykjavík, Iceland.}}

\date{December 1, 2025}

\begin{document}

\maketitle

\setcitestyle{authoryear,round,semicolon}

\begin{abstract}

Stability is a key property of both forward models and inverse problems, and depends on the norms considered in the relevant function spaces. For instance,  stability estimates for hyperbolic partial differential equations are often based on energy conservation principles, and are therefore expressed in terms of $L^2$ norms. The focus of this paper is on stability with respect to the $L^\infty$ norm, which is more relevant to detect localized phenomena.  The linear wave equation is not stable in  $L^\infty$, and we design an alternative solution method based on the regularization of Fourier multipliers, which is stable in $L^\infty$. Furthermore, we show how these ideas can be extended to inverse problems, and design a regularization method for the inversion of compact operators that is stable in $L^\infty$. We also discuss the connection with the stability of deep neural networks modeled by hyperbolic PDEs.

\end{abstract}

\setcitestyle{numbers,square,comma}


\section{Introduction}\label{sec:intro}

One of the key mathematical questions regarding the study of ordinary differential equations (ODEs) and of partial differential equations (PDEs) is the continuous dependence of the solutions on the initial data. This property guarantees that the solution to a certain ODE or PDE is stable with respect to small perturbations of the initial data. The absence of this stability is at the core of chaos theory, which studies those dynamical systems that are highly sensitive to the initial conditions: a very small change can lead to a very different outcome. However, many physical models do satisfy this stability property.
As simple examples, we can consider the linear heat or wave equations, which are second-order PDE: for these problems, such continuous dependence holds, and because of the linearity, it can be expressed by the boundedness of the linear operator
\[
\text{initial conditions $\;\longmapsto\;$ solution.}
\]
As a consequence, we have Lipschitz stability of the solutions with respect to the initial conditions.

Let us focus on hyperbolic equations. One of the key tools employed to derive stability  estimates is the energy conservation principle \cite{dautray-lions-5,evans2022partial}. As a result, these estimates are written in terms of $L^2$-based norms. 
In order to treat general $L^p$ norms, different methods are needed. For example, in the case of the linear wave equation with constant coefficients, the stability estimates correspond to the boundedness of Fourier multipliers between suitable $L^p$ spaces \cite{PERAL1980114,sogge1993p}.   In addition to the theoretical motivation, considering different norms may be relevant in the applications, as we will discuss below. In particular, the $L^\infty$ norm quantifies features (like localized spikes) that may be qualitatively different from those measured by the $L^2$ norm (see e.g.\  Figs.~\ref{fig:weq-pert} and \ref{fig:weq-signalpert} below).

In this work, we consider the linear wave equation with constant coefficient as a prototypical example. Stability in $L^\infty$ does not hold: if we consider the $\|\cdot\|_\infty$ norm, a small perturbation of the initial condition $u_0$ can lead to arbitrarily large deviations at time $t=1$. In order to overcome this issue, we introduce a regularized solution of the PDE. This is constructed by means of a double filter (in space and in frequency) of the corresponding Fourier multiplier. Intuitively, and in simple terms, this gives rise to a family of regularized solution operators $B_{\alpha,\beta}$  that are bounded in $L^\infty$ and that approximate, as $\alpha,\beta\to 0$, the true solution $u$, namely, $B_{\alpha,\beta} u_0 \to u$ in $L^\infty$.

Further, we extend these ideas to inverse problems, in which a similar issue arises:
 classical regularization techniques give rise to reconstruction algorithms that are bounded in the relevant Hilbert space norms, such as $L^2$, but may be unbounded with respect to other norms, such as $L^\infty$. Regularization methods in Banach spaces have been studied extensively, but mostly with a focus on the case $p=1$, promoting sparsity (see \cite{SchusterKaltenbacherHofmannKazimierski+2012} and the references therein). In this paper, we take inspiration from the approach developed for the wave equation and design a regularization method for inverse problems that is stable in $L^\infty$. This is achieved by considering a nonstandard spectral filtering of the singular value decomposition.

The results of this paper are relevant for the study of the stability of deep neural networks (DNNs). The instability of DNNs is linked to the existence of \emph{adversarial examples}, first observed in the context of image recognition with deep neural networks \cite{biggio2013evasion,szegedy2013intriguing,goodfellow2014explaining,moosavi2016deepfool,madry2017towards,carlini2017towards}.
These are images that have been perturbed intentionally and imperceptibly so that the network makes an incorrect prediction.
More recently, and naturally, adversarial examples have also found their way to the domain of deep learning for inverse problems \cite{antun2020instabilities,genzel2020solving,gottschling2020troublesome}.
While  these perturbations are usually quantified in terms of the $L^2$ norm, several works have studied perturbations measured in different ways \cite{xiao2018spatially,alaifari2019adef,alaifari2023localized}, which may be qualitatively more relevant for human perception.

When both the layers and the spatial variable are seen as continuous, the action of a residual convolutional neural network (CNN) can be seen as the input-to-solution map associated to a nonlinear time-dependent PDE \cite{haber-ruthotto-2020}. In particular, hyperbolic CNNs were introduced by using hyperbolic PDEs with the aim of obtaining an energy preserving propagation and, consequently, a stable network with respect to the $L^2$ norm. The derivations in this paper show, for a simpler linear problem, that direct extensions to the $L^\infty$ norm, which is more adapted to localized perturbations, are not possible. Furthermore, the proposed regularized solution offers a possible way to stabilize the network. The extension to more complicated nonlinear PDEs is needed to fully capture the nonlinearity of DNNs, but cannot be treated with the current methods; it is an interesting avenue for future research.

The paper is structured as follows. In Section~\ref{sec:deep} we expand upon the connections to the stability of deep neural networks. In Section~\ref{sec:adv} we introduce the notion of adversarial example for linear operators, focusing on the role played by the different norms involved. In Section~\ref{sec:wave} we consider, as motivating example, the linear wave equation, and propose a method for the $L^\infty$ regularization of Fourier multipliers. In Section~\ref{sec:IP}, we adapt these ideas for the regularization of inverse problems. Finally, some proofs are collected in \ref{sec:appendix_proofs}.

\section{Connections to the stability of deep neural networks}\label{sec:deep}

We begin by briefly discussing the link between feed-forward deep neural networks (DNN) and hyperbolic equations, and why the results of this paper can provide interesting insights for the study of the stability of DNNs. 

Let us start by introducing DNNs; the reader is referred to \cite{Goodfellow-et-al-2016,higham-etal-2019} for additional details on the basics of DNNs. A DNN is a function  $f_\theta\colon\R^{n_0}\to\R^{n_N} $ that can be written as a composition of $N$ layers:
\[
f_\theta(u_0)=u_N,
\]
where
\begin{equation}\label{eq:dnn-layer}
u_{j}=f_{\theta_j}(u_{j-1}),\qquad j=1,\dots,N.
\end{equation}
Here, $f_{\theta_j}\colon\R^{n_{j-1}}\to\R^{n_{j}}$ describes the action of the $j$-th layer, and is a map parametrized by $\theta_j\in\R^{p_j}$. We collect all the weights $\theta_j$ into a single vector $\theta\in\R^p$, which is typically learned during training. The most common choice for $f_{\theta_j}$ is 
\[
f_{\theta_j}(s) = \sigma(W_j s+b_j),
\]
where $W_j\in\R^{n_j\times n_{j-1}}$, $b_j\in \R^{n_j}$ is a bias term, and $\sigma\colon\R\to\R$ is a fixed nonlinearity, such as the rectified linear unit $\sigma(y)=\max(y,0)$, acting pointwise. In the case when no additional structure is imposed on the weights, the free parameters are $\theta_j = (W_j,b_j)$, and the network is called fully connected, because every node of the output is linked to every node in the input through the linear map $W_j$. It is also possible to restrict the choice of the maps $W_j$: for example, considering convolutions allows for a reduction in the number of free parameters, and yields translation invariant (often called equivariant) layers.  
An alternative to \eqref{eq:dnn-layer} was proposed in \cite{ResNets} in order to learn a perturbation of the identity only, namely
\begin{equation}\label{eq:resnet-layer}
u_{j}=u_{j-1}+f_{\theta_j}(u_{j-1}),\qquad j=1,\dots,N.
\end{equation}
Here, for simplicity we consider only the case when all the dimensions $n_j$ coincide, and $f_\theta,f_{\theta_j}\colon\R^d\to\R^d$. This choice gives rise to the so-called residual neural networks (ResNets).

It is possible to view the discrete variable $j$ parametrizing the layers of a ResNet as a discretization of the continuous time of a dynamical system \cite{e-2017,haber-ruthotto-2018,neuralodes-2018}. More precisely, \eqref{eq:resnet-layer} can be seen as the forward Euler step (with step size $1$) of the Cauchy problem
\begin{equation}\label{eq:cauchy}
    \left\{
\begin{array}{l}
    \partial_t u(t) =f_{\theta_t}(u(t)),\qquad t\in (0,T],  \\
     u(0)=u_0.
\end{array}
    \right.
\end{equation}
The output of the network is given by the solution at time $t=T$, namely
$f_\theta(u_0) = u(T)$. Furthermore, we can view the vectors $u\in\R^d$ as discretizations of functions $u$ of two or three variables, modeling 2D and 3D signals, respectively. In this continuous settings, convolutions with filters with small support become differential operators \cite{haber-ruthotto-2020}. As a consequence, the maps $f_{\theta_t}$ can be taken as certain, possibly nonlinear, differential operators, acting on the function $u$, and \eqref{eq:cauchy} becomes an initial value problem for a PDE, called parabolic CNN in \cite{haber-ruthotto-2020}. The simplest example in this setting is associated to $f_{\theta_t} u=\Delta u$, giving rise to the (linear) heat equation. Considering time/spatial-dependent, possibly nonlinear, differential operators with learnable parameters gives rise to more complicated evolutions.

Heat-type equations tend to dissipate energy. For energy preservation, it is natural to consider hyperbolic CNNs, which take the form
\begin{equation}\label{eq:hyperbolic}
    \left\{
\begin{array}{l}
    \partial^2_t u(t) =f_{\theta_t}(u(t)),\qquad t\in (0,T],  \\
     u(0)=u_0.
\end{array}
    \right.
\end{equation}
As above, the trivial choice $f_{\theta_t} u=\Delta u$ corresponds to the (linear) wave equation with constant coefficients, and other choices give rise to more complicated, possibly nonlinear, hyperbolic PDEs.

The theoretical results of \cite{haber-ruthotto-2020} show that both PDEs \eqref{eq:cauchy} and \eqref{eq:hyperbolic} have Lipschitz stability properties, in the sense that
\begin{equation*}
    \|u^{(1)}(T)-u^{(2)}(T)\|_2\le C \|u^{(1)}_0-u^{(2)}_0\|_2
\end{equation*}
for some $C>0$.
These estimates are useful in the context of stability of DNNs and, especially, for studying robustness against adversarial attacks, as discussed in Section~\ref{sec:intro}. Indeed, they guarantee that the output of a neural network depends continuously (in fact, in a Lipschitz way) on the input, and they quantify this dependence by the constant $C$. However, as we argue in Section~\ref{sec:intro}, the $L^2$ norm may not capture all the relevant deformations of a signal. For instance, the $L^\infty$ norm is more suited to capture localized perturbations. The results of this paper, and especially those presented in Section~\ref{sec:wave} address this issue in the very simple setting of the linear wave equation: stability in $L^\infty$ does not hold in general, and a filtering approach is introduced to overcome this issue. Our approach does not directly apply to more complicated nonlinear settings, which are interesting for applications to neural networks, but can serve as an initial investigation on this issue.

\section{Adversarial sequences for linear operators}\label{sec:adv}

The central theme of this work is that a bounded linear operator between Banach spaces may not be bounded with respect to other norms that are of interest.
Vectors in the domain of the operator that reveal this behaviour will be termed \emph{adversarial}.
In this section, we define this term formally in an abstract setting.

A \emph{Banach couple} is a pair of Banach spaces such that each can be continuously embedded in a common Hausdorff topological vector space.
The archetypal Banach couples are pairs of \(L^p\)-spaces, which are all continuously embedded in the space of (equivalence classes of) Lebesgue-measurable functions with the topology of local convergence in measure \cite{kalton2003interpolation}.
With an abuse of notation, we interpret all set operations on Banach couples as operations in the ambient topological vector space under the continuous embeddings.

\begin{definition}[Adversarial perturbations and sequences]
\label{def:weq:aps}
    Let \((X,X')\) and \((Y,Y')\) be two Banach couples, and let \(A\colon X\to Y\) be a bounded linear operator.
    Let \(r\in X\cap X'\) and let \((r_n)_{n\in\N}\) be a sequence in \(X\cap X'\).
    \begin{enumerate}[(i)]
        \item If \(Ar\notin Y'\), i.e., \(\norm{Ar}_{Y'}=\infty\), then we say that \(r\) is an \emph{adversarial perturbation} for \(A\) (relative to \(X'\) and \(Y'\)).
        \item If \((r_n)_{n\in\N}\) is bounded in both \(X\) and \(X'\), and \((Ar_n)_{n\in\N}\) is unbounded in \(Y'\),
        then we say that \((r_n)_{n\in\N}\) is an \emph{adversarial sequence} for \(A\) (relative to \(X'\) and \(Y'\)).
    \end{enumerate}
\end{definition}

\begin{example}
     Take \(X=X'=Y=L^2(\R)\), \(Y'=L^\infty(\R)\), and let \(A\) be 
    the Fourier transform \(A=\fourier\colon X\to Y\).
    Any square integrable function with an unbounded Fourier transform is an adversarial perturbation for \(\fourier\).
    Define the functions
    \(\rho_n=n\characteristic{[0,n^{-2}]}\)
    and \(r_n=\fourier^{-1}\rho_n\), \(n\in\N\).
    Then \(\norm{r_n}_{L^2(\R)}=\norm{\rho_n}_{L^2(\R)}=1\) and \(\norm{\rho_n}_{L^\infty(\R)}=n\) for all \(n\in\N\), so \((r_n)_{n\in\N}\) is an adversarial sequence.
\end{example}

Although adversarial perturbations exhibit a more dramatic effect, it can be convenient to construct adversarial sequences instead.
The following consequence of the closed graph theorem shows that existence of the latter implies existence of the former.

\begin{theorem}
\label{thm:advseqadvpert}
    Let \((X,X')\) and \((Y,Y')\) be two Banach couples
    and let \(A\colon X\to Y\) be a bounded linear operator.
    If there exists an adversarial sequence for \(A\) (relative to \(X'\) and \(Y'\)),
    then there exists an adversarial perturbation for 
    \(A\) (relative to \(X'\) and \(Y'\)).
\end{theorem}
\begin{proof}
See \ref{sec:appendix_proofs}.
\end{proof}

\begin{example}
\label{ex:domainchangeintegral}
    Theorem \ref{thm:advseqadvpert}
    does not hold if we lift the requirement in Definition \ref{def:weq:aps} 
    that adversarial sequences are bounded in the domain of the operator.
    Take
    \begin{equation*}
    X=L^1(\R),\quad
    X'=L^\infty(\R),\quad
    Y=Y'=\R,
    \end{equation*}
    and let \(A\) be the integration operator
    \begin{equation*}
        A\colon X\to Y,\ f\mapsto \int f.
    \end{equation*}
    Then \(A\) is bounded, as \(\norm{Af}_Y=\abs{\int f}\leq \int\abs{f}=\norm{f}_X\),
    and 
    it is not bounded with respect to \(\norm{\cdot}_{X'}\) on the intersection \(X\cap X'\)
    since the functions \(\characteristic{[0,n]}\in X\cap X'\) satisfy 
    \begin{equation*}
    \smallnorm{\characteristic{[0,n]}}_{X'}=1
    \quad\text{and}\quad
    \smallnorm{A\characteristic{[0,n]}}_{Y'}=\abs{\int\characteristic{[0,n]}}=n   
    \end{equation*}
    for all \(n\in \N\).
    However, no adversarial perturbation exists for \(A\) since \(\abs{Ar}<\infty\) for all \(r\in X\cap X'\).
    This example does not violate Theorem \ref{thm:advseqadvpert} because the sequence \((\characteristic{[0,n]})_{n\in\N}\) is not bounded in \(X\), and thus is not an adversarial sequence.
\end{example}

\section{\texorpdfstring{$L^\infty$}{L infty} stability of the wave equation  via  Fourier multipliers}\label{sec:wave}

In this section, we consider a Fourier multiplier operator 
\(B\colon L^2(\R^d)\to L^2(\R^d)\),
\begin{equation}\label{eq:multipliermu}
    Bf = 
    \fourier^{-1}\left(
    \mu\cdot \fourier{f}
    \right)
\end{equation}
with some \(\mu\in L^\infty(\R^d)\). Here, $\mathcal F\colon L^2(\R^d)\to L^2(\R^d)$ denotes the Fourier transform with the normalization $\mathcal{F}f(\xi)=\widehat f(\xi)=\int_{\R^d} f(x)e^{-2\pi i x\cdot \xi}\,dx$ for $f\in L^1(\R^d)\cap L^2(\R^d)$.
Such an operator $B$ is bounded, with \(\norm{B}_{L^2(\R^d)\to L^2(\R^d)}\leq \norm{\mu}_{L^\infty(\R^d)}\),
but may be unbounded when viewed under a different norm.
We wish to approximate \(B\) with an operator that is bounded in a stronger sense,
for example as an operator \(L^\infty(\R^d)\to L^\infty(\R^d)\) in addition to \(L^2(\R^d)\to L^2(\R^d)\).

\subsection{Instability in wave propagation}
\label{sec:weqzeroinit}
As a motivating example, consider the wave equation
\begin{align*}
    \partial_t^2 u(t) - \Delta u(t) & = 0\\
    u(0) & = f\\
    \partial_tu(0) & = 0,
\end{align*}
in \(\R^3\),
and the operator \(B\) which propagates the initial condition \(f\) up to time \(t=1\), namely, $Bf = u(1)$. Here, $u$ is a function of two variables, $t$ and $x$, and with an abuse of notation we write $u(t)$ for the function $u(t,\cdot)$.
Viewing the wave equation under the spatial Fourier transform,
\begin{equation*}
    \fhat{u}(t) = \cos\left(2\pi\enorm{\cdot}t\right)\fhat{f},\quad t\geq 0,
\end{equation*}
we verify that \(B\) is a Fourier multiplier of the form \eqref{eq:multipliermu} with $\mu(\xi) = \cos\left(2\pi |\xi|\right)$, hence a bounded linear operator \(L^2(\R^3)\to L^2(\R^3)\).
However,
it is simple to find bounded initial conditions that can become arbitrarily large in \(L^\infty(\R^3)\) at time \(t=1\).
For example,
consider smooth radially symmetric initial conditions,
\begin{equation*}
    f(x) = g(\enorm{x}),\quad g\colon [0,\infty)\to\R.
\end{equation*}
When evaluated at \(x=0\), the solution \(u\) satisfies
\begin{equation*}
    u(t,0)
    = g(t) + t\cdot g'(t)
\end{equation*}
(see e.g.\ Kirchhoff's formula \cite[\S 12]{vladimirov1971}).
Thus, by increasing the slope of \(g\) at \(\enorm{x}=1\), we can increase the \(L^\infty(\R^3)\) norm of \(Bf=u(1)\).
We construct an adversarial sequence (cf.\ Definition \ref{def:weq:aps}) relative to \(L^\infty(\R^3)\)
by setting,
for all \(n\in\N\),
\(f_n(x)=g_n(\enorm{x})\) with $g_n$ compactly supported and such that \(\abs{g_n(r)}\leq 1\), \(r\geq 0\),  as well as \(\abs{g_n'(1)}\xrightarrow[n\to\infty]{}\infty\).
Then, we get
\begin{equation*}
   \|Bf_n\|_{L^\infty(\R^3)}\ge \abs{Bf_n(0)} = \abs{g_n(1) + g_n'(1)}\xrightarrow[n\to\infty]{} \infty.
\end{equation*}
See Fig.~\ref{fig:weq-pert} for an example. Figure~\ref{fig:weq-signalpert} shows such a sequence as perturbations to a signal of larger amplitude.

\begin{figure}
    \centering
    \includegraphics[width=\textwidth]{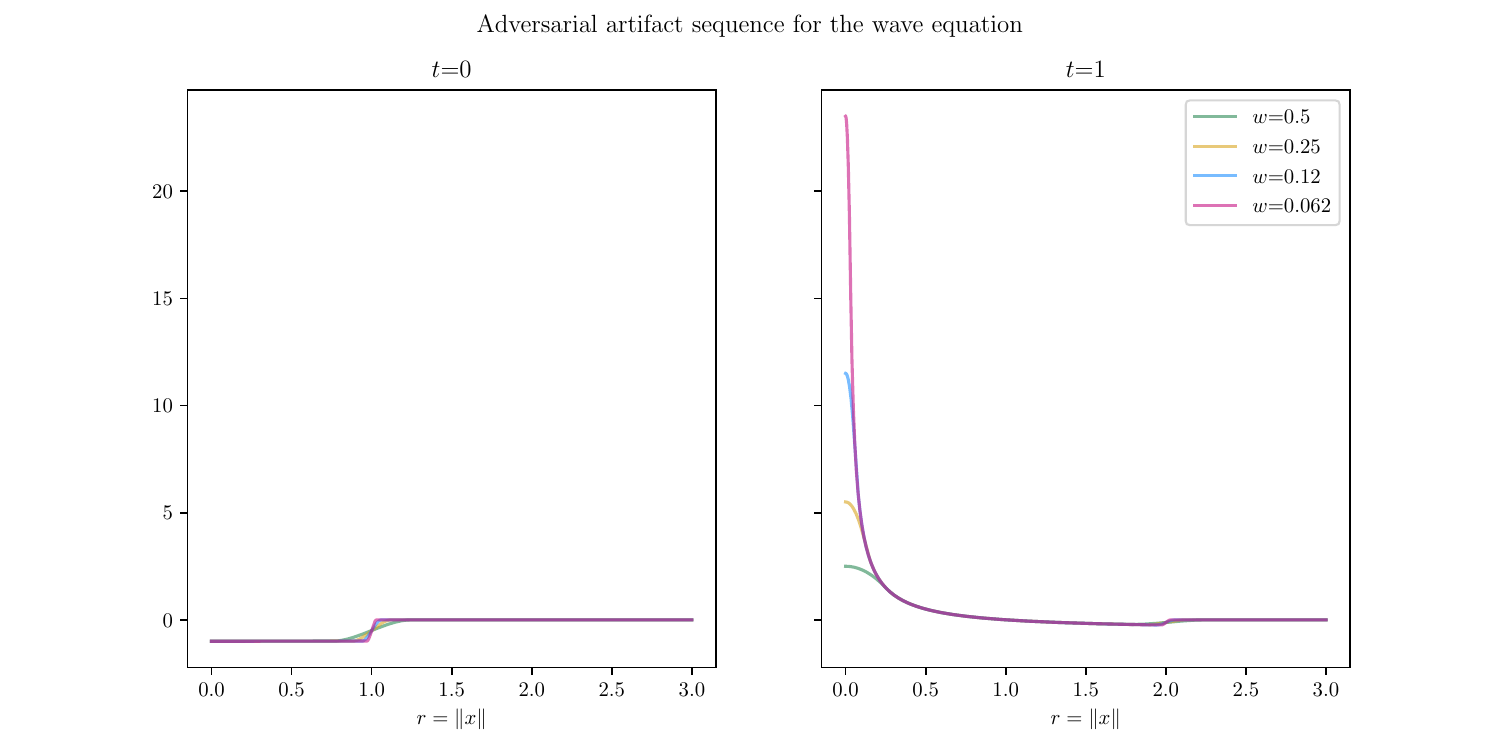}
    \caption{Radial component of spherical waves at initial time \(t=0\) (left) and end time \(t=1\) (right). The initial states \(f_n\), \(n=1,2,3,4\), are continuously differentiable functions that transition from the value \(-1\) to the value 0 in a window of width \(w=2^{-n}\), centered at \(r=1\). The increasingly steep transition causes the end states \(Bf_n\) to grow at \(r=0\) as \(n\) increases.}
    \label{fig:weq-pert}
\end{figure}

\begin{figure}
    \centering
    \includegraphics[width=\textwidth]{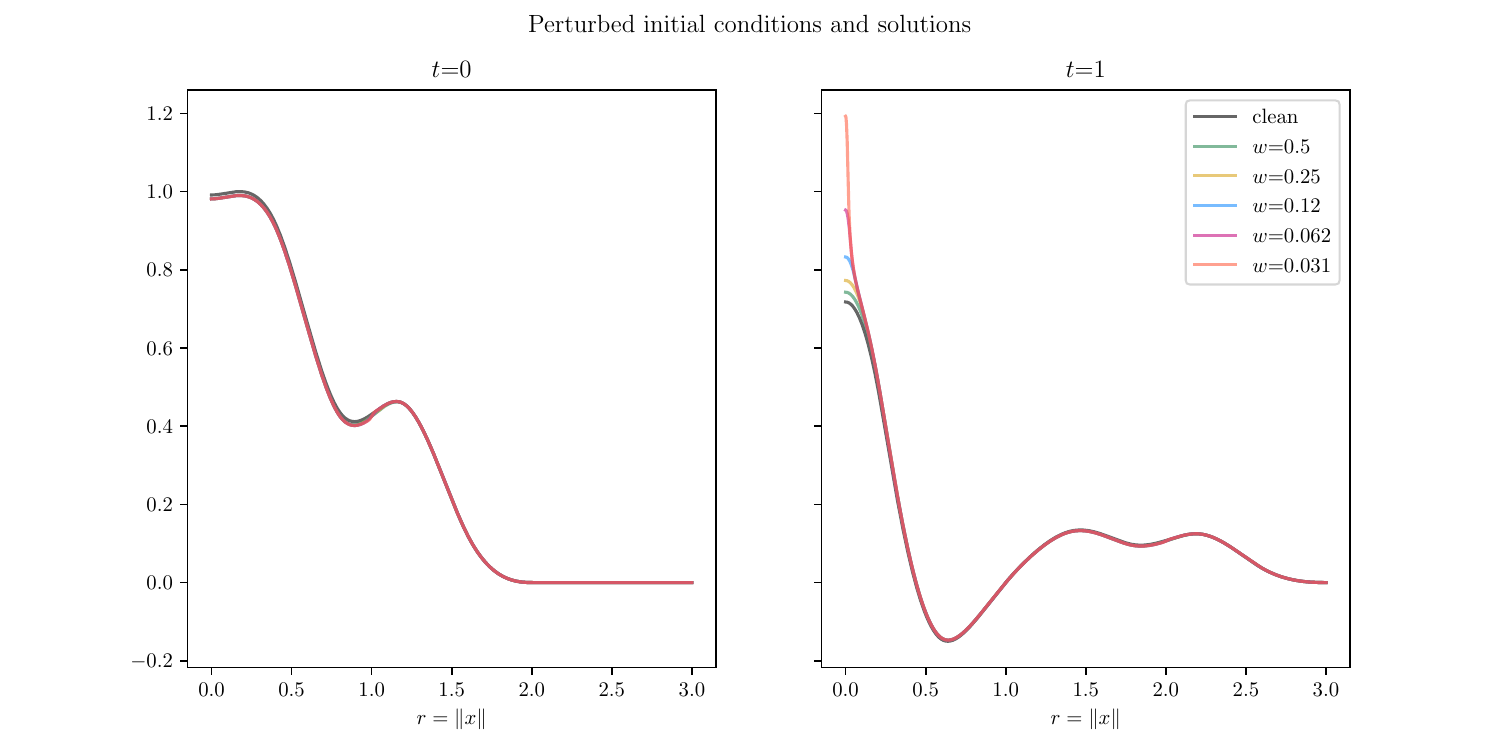}
    \caption{Radial component of spherical waves at initial time \(t=0\) (left) and end time \(t=1\) (right). The ``clean'' initial state is perturbed by $0.01\cdot f_n$, \(n=1,\ldots,5\), where \(f_n\) are the functions from Fig. \ref{fig:weq-pert}.}
    \label{fig:weq-signalpert}
\end{figure}

\subsection{Filtering in space and frequency}
\label{sec:regfouriermult}

In order to mitigate effects such as the adversarial sequences in Section \ref{sec:weqzeroinit},
we can employ filters in both frequency and space to approximate \(B\).

\begin{proposition}
    Let \((\kappa_\alpha)_{\alpha>0}\)
    and
    \((h_\beta)_{\beta>0}\)
    be two families of functions such that 
    \(\kappa_\alpha, h_\beta \in L^2(\R^d) \cap L^\infty(\R^d)\)
    for all
    \(\alpha,\beta >0\).
    Then the operator
    \begin{equation*}
        B_{\alpha,\beta}\colon
        L^p(\R^d)\to L^p(\R^d),
        \quad
        f\mapsto
        \fourier^{-1}\left(
        \mu\cdot 
        \kappa_\alpha
        \cdot 
        \fourier\left({h_\beta} f\right)
        \right)        
    \end{equation*}
    is well defined and bounded for every \(2\leq p\leq \infty\). In particular,
    \begin{align*}
        \|B_{\alpha,\beta}\|_{L^2(\R^d)\to L^2(\R^d)}\le \norm{
    \mu
    }_{L^\infty(\R^d)}
    \norm{
    \kappa_\alpha
    }_{L^\infty(\R^d)}
    \norm{
    {h_\beta}
    }_{L^\infty(\R^d)},\\
    \|B_{\alpha,\beta}\|_{L^\infty(\R^d)\to L^\infty(\R^d)}\le
    \norm{
    \mu
    }_{L^\infty(\R^d)}
    \norm{
    \kappa_\alpha
    }_{L^2(\R^d)}
    \norm{
    {h_\beta}
    }_{L^2(\R^d)}.
    \end{align*}
\end{proposition}
\clearpage
\begin{proof}
Straightforward computation,
using
H\"older's inequality and
the Plancherel theorem,
gives
\begin{align*}
    \norm{B_{\alpha,\beta}f}_{L^2(\R^d)}
    &=
    \norm{\mu
    \kappa_\alpha
    \fourier\left({h_\beta} f\right)
    }_{L^2(\R^d)}\\
    &\leq
    \norm{
    \mu
    }_{L^\infty(\R^d)}
    \norm{
    \kappa_\alpha
    }_{L^\infty(\R^d)}
    \norm{
    \fourier\left({h_\beta} f\right)
    }_{L^2(\R^d)}\\
    &=
    \norm{
    \mu
    }_{L^\infty(\R^d)}
    \norm{
    \kappa_\alpha
    }_{L^\infty(\R^d)}
    \norm{
    {h_\beta} f
    }_{L^2(\R^d)}\\
    &\leq
    \norm{
    \mu
    }_{L^\infty(\R^d)}
    \norm{
    \kappa_\alpha
    }_{L^\infty(\R^d)}
    \norm{
    {h_\beta}
    }_{L^\infty(\R^d)}
    \norm{
    f
    }_{L^2(\R^d)},
\end{align*}
and further,
\begin{align*}
    \norm{B_{\alpha,\beta}f}_{L^\infty(\R^d)}
    &\leq
    \norm{\mu
    \kappa_\alpha
    \fourier\left({h_\beta} f\right)
    }_{L^1(\R^d)}\\
    &\leq
    \norm{
    \mu
    }_{L^\infty(\R^d)}
    \norm{
    \kappa_\alpha
    \fourier\left({h_\beta} f\right)
    }_{L^1(\R^d)}\\
    &\leq
    \norm{
    \mu
    }_{L^\infty(\R^d)}
    \norm{
    \kappa_\alpha
    }_{L^2(\R^d)}
    \norm{
    \fourier\left({h_\beta} f\right)
    }_{L^2(\R^d)}\\
    &=
    \norm{
    \mu
    }_{L^\infty(\R^d)}
    \norm{
    \kappa_\alpha
    }_{L^2(\R^d)}
    \norm{
    {h_\beta} f
    }_{L^2(\R^d)}\\
    &\leq
    \norm{
    \mu
    }_{L^\infty(\R^d)}
    \norm{
    \kappa_\alpha
    }_{L^2(\R^d)}
    \norm{
    {h_\beta}
    }_{L^2(\R^d)}
    \norm{
    f
    }_{L^\infty(\R^d)}.
\end{align*}
Therefore \(B_{\alpha,\beta}\) is bounded 
both as an
operator \(L^2(\R^d)\to L^2(\R^d)\)
and
\(L^\infty(\R^d)\to L^\infty(\R^d)\).
By the Riesz-Thorin interpolation theorem
it is therefore bounded as an operator \(L^p(\R^d)\to L^p(\R^d)\) for \(p\in[2,\infty]\).
\end{proof}

One can, for example, take as filter the families
\begin{equation*}
    \kappa_\alpha = {h_\alpha} = \characteristic{\ballorigin{\alpha^{-1}}},
\end{equation*}
(where \(\ballorigin{\delta}\) is the  ball around the origin with radius \(\delta>0\)).
The bounds above become
\begin{equation*}
    \norm{B_{\alpha,\alpha}f}_{L^2(\R^d)}\leq
    \norm{
    \mu
    }_{L^\infty(\R^d)}
    \norm{
    f
    }_{L^2(\R^d)},
\end{equation*}
and
\begin{equation*}
    \norm{B_{\alpha,\alpha}f}_{L^\infty(\R^d)}\leq
    V_d\alpha^{-d}
    \norm{
    \mu
    }_{L^\infty(\R^d)}
    \norm{
    f
    }_{L^\infty(\R^d)}
\end{equation*}
where \(V_d\) is the volume of the unit ball in \(\R^d\).
We shall see that this choice is not optimal.

For certain families of filters,
the operator \(B_{\alpha,\beta}\) can be written as a composition of \(B\) with a preconditioning operator \(T_{\alpha,\beta}\).
This is highly beneficial for implementation if the multiplier \(\mu\) is not known.
Let us record some boundedness properties of this operator.

\begin{lemma}
\label{thm:preconbounded}
Let \((k_\alpha)_{\alpha>0}\)
and
\((h_\beta)_{\beta>0}\)
be two families of functions such that 
\(k_\alpha\in L^1(\R^d)\)
for all
\(\alpha>0\), and \({h_\beta}\in
L^2(\R^d)\cap L^\infty(\R^d)\) for all \(\beta>0\).
The linear operator
\begin{equation}
    \label{eq:preconbounded}
    T_{\alpha,\beta}\colon L^p(\R^d)\to L^p(\R^d)\cap L^q(\R^d)
    ,
    \quad
    f\mapsto k_\alpha * (h_\beta\cdot f),
\end{equation}
is well defined and bounded for \(1\leq p \leq \infty\),
where
\begin{equation*}
    q = \begin{cases}
        2 p/(p + 2) \in [1,2],& \text{if}\ p>2\\
        p,&\text{if}\ 1\leq p \leq 2.
    \end{cases}
\end{equation*}

\end{lemma}
\begin{proof}
    By Young's convolution inequality,
    we have
    \begin{align*}
        \norm{k_\alpha * (h_\beta f)}_{L^p(\R^d)}
        &\leq
        \norm{k_\alpha}_{L^1(\R^d)}\norm{h_\beta f}_{L^p(\R^d)}
        \\&
        \leq
        \norm{k_\alpha}_{L^1(\R^d)}\norm{h_\beta}_{L^\infty(\R^d)}\norm{f}_{L^p(\R^d)}
    \end{align*}
    for \(f\in L^p(\R^d)\), which 
    shows that \(T_{\alpha,\beta}\) is bounded \(L^p(\R^d)\to L^p(\R^d)\).
    Now, for \(2\leq p \leq \infty,\)
    \(q = 2p/(p+2)\)
    we have 
    \begin{align*}
        \norm{k_\alpha * (h_\beta f)}_{L^q(\R^d)}
        &\leq
        \norm{k_\alpha}_{L^1(\R^d)}\norm{h_\beta f}_{L^q(\R^d)}
        \\&
        \leq
        \norm{k_\alpha}_{L^1(\R^d)}\norm{h_\beta}_{L^2(\R^d)}\norm{f}_{L^p(\R^d)}
    \end{align*}
    by applying Young's convolution inequality again, together with H\"older's inequality.
    This shows that \(T_{\alpha,\beta}\) is bounded \(L^p(\R^d)\to L^q(\R^d)\) as well.
\end{proof}
An interpretation of Lemma \ref{thm:preconbounded} is that 
for any \(1\leq p \leq \infty\), the operator \(T_{\alpha,\beta}\colon L^p(\R^d)\to L^p(\R^d)\) is bounded and maps into the domain of the Fourier transform.
In fact, a corollary of Lemma \ref{thm:preconbounded} is that \(T_{\alpha,\beta}\colon L^p(\R^d)\to L^2(\R^d)\) is bounded for all \(2\leq p \leq \infty\).

\begin{remark}
    If \(\kappa_\alpha = \fourier k_\alpha\) with \(k_\alpha\in L^1(\R^d)\cap L^2(\R^d)\) for all \(\alpha >0\),
    then \(\kappa_\alpha\in L^2(\R^d)\cap L^\infty(\R^d)\) and we get
    \begin{equation*}
    B_{\alpha,\beta} = BT_{\alpha,\beta}
    \end{equation*}
    by the convolution theorem:
    \begin{align*}
        B_{\alpha,\beta}f
        &=
        \fourier^{-1}\left(
        \mu\cdot 
        \kappa_\alpha
        \cdot 
        \fourier\left({h_\beta}\cdot f\right)
        \right)
        \\&
        =
        \fourier^{-1}\left(
        \mu\cdot 
        \fourier(k_\alpha)
        \cdot 
        \fourier\left({h_\beta}\cdot f\right)
        \right)
        \\&
        =
        \fourier^{-1}\left(
        \mu\cdot 
        \fourier\left(k_\alpha * ({h_\beta}\cdot f)\right)
        \right)
        \\&
        =
        BT_{\alpha,\beta}f.
    \end{align*}
    Note that Lemma \ref{thm:preconbounded}, H\"older's inequality, and the Hausdorff-Young inequality  ensure that the Fourier transforms of \(T_{\alpha,\beta}f\) and \(h_\beta\cdot f\) exist for \(f\in L^p(\R^d)\) for any \(1\leq p \leq \infty\).
    However, in this case, \(\kappa_\alpha\) is continuous by the Riemann-Lebesgue lemma,
    so filters such as
    \(\kappa_\alpha = \characteristic{\ballorigin{\alpha^{-1}}}\) are excluded.
\end{remark}

\begin{theorem}[Pointwise approximation in \(L^2(\R^d)\)]
\label{thm:regfouriermult_pwl2approx}
    Let \(k,\fhat{h}\in L^1(\R^d)\cap L^2(\R^d)\) be two functions satisfying \(\int k = \int \fhat{h} = 1\).
    Define the families \((k_\alpha)_{\alpha>0}\) and \((h_\beta)_{\beta>0}\)
    through dilation,
    \begin{align}
    \label{eq:approxiddilate_k}
        k_\alpha(x) &= \alpha^{-n} k(x/\alpha)\\
    \label{eq:approxiddilate_h}
        \fhat{h}_\beta(x) &= \beta^{-n} \fhat{h}(x/\beta),\qquad h_\beta = \mathcal{F}^{-1}(\fhat{h}_\beta).
    \end{align}
    Then, for every \(f\in L^2(\R^d)\) we have
    \begin{equation*}
        \norm{T_{\alpha,\beta}f - f}_{L^2(\R^d)}\xrightarrow[\alpha,\beta\to 0]{} 0
    \end{equation*}
    and
    \begin{equation*}
        \norm{B_{\alpha,\beta}f - Bf}_{L^2(\R^d)}\xrightarrow[\alpha,\beta\to 0]{} 0.
    \end{equation*}
\end{theorem}
\begin{proof}
    Note that \(h_\beta = \fourier^{-1}[\fhat{h}_\beta]\in L^2(\R^d)\cap L^\infty(\R^d)\).
    We have established that \(B_{\alpha,\beta} = BT_{\alpha,\beta}\)
    and that \(\norm{B}_{L^2(\R^d)\to L^2(\R^d)}\leq \norm{\mu}_{L^\infty(\R^d)}\). 
    Therefore   
    \begin{equation*}
        \norm{B_{\alpha,\beta}f - Bf}_{L^2(\R^d)}
        \leq
        \norm{\mu}_{L^\infty(\R^d)}
        \norm{T_{\alpha,\beta}f - f}_{L^2(\R^d)}.
    \end{equation*}
    Using the triangle inequality, 
    Young's convolution inequality,
    and the fact that \(\norm{k_\alpha}_{L^1(\R^d)} = \norm{k}_{L^1(\R^d)}\),
    we get
    \begin{align*}
        \norm{T_{\alpha,\beta}f - f}_{L^2(\R^d)}
        & =
        \norm{ k_\alpha * (h_\beta f) - f }_{L^2(\R^d)}
        \\
        &\leq 
        \norm{k_\alpha * (h_\beta f) - k_\alpha * f}_{L^2(\R^d)}
        +
        \norm{k_\alpha*f - f}_{L^2(\R^d)}\\
        &=
        \norm{k_\alpha * (h_\beta f- f)}_{L^2(\R^d)}
        +
        \norm{k_\alpha*f - f}_{L^2(\R^d)}
        \\
        &\leq
        \norm{k_\alpha}_{L^1(\R^d)}
        \norm{h_\beta f - f}_{L^2(\R^d)}
        +
        \norm{k_\alpha * f - f}_{L^2(\R^d)}
        \\
        &=
        \norm{k}_{L^1(\R^d)}
        \norm{\fourier(h_\beta f - f)}_{L^2(\R^d)}
        +
        \norm{k_\alpha * f - f}_{L^2(\R^d)}
        \\
        &=
        \norm{k}_{L^1(\R^d)}
        \norm{\fhat{h}_\beta * \fhat{f} - \fhat{f}}_{L^2(\R^d)}
        +
        \norm{k_\alpha * f - f}_{L^2(\R^d)}.
    \end{align*}
    Since \((k_\alpha)_{\alpha>0}\) and \((\fhat{h}_\beta)_{\beta>0}\) form an approximation to the identity,
    we have
    \begin{align*}
        \norm{k_\alpha * f - f}_{L^2(\R^d)}
        &\xrightarrow[\alpha\to 0]{}0\quad\text{and}\\
        \norm{\fhat{h}_\beta * \fhat{f} - \fhat{f}}_{L^2(\R^d)}
        &\xrightarrow[\beta\to 0]{}0,
    \end{align*}
    which concludes the proof.
\end{proof}

\begin{remark}
    The choice of filters (\ref{eq:approxiddilate_k}-\ref{eq:approxiddilate_h})
    is equivalent to
    \begin{align*}
        \kappa_\alpha(\xi) &= \kappa(\alpha \xi)
        \\
        h_\beta(x) &= h(-\beta x),
    \end{align*}
    with \(h = \fourier^{-1}\fhat{h}\) and \(\kappa = \fourier k\).
    Note that an integrable \(\fhat{h}\) yields a continuous \(h\), so we have now excluded both \(\kappa_\alpha=\characteristic{\ballorigin{\alpha^{-1}}}\) and \(h_\beta=\characteristic{\ballorigin{\beta^{-1}}}\).
    We record the norms of the filters as follows:
\begin{align*}
    \smallnorm{\kappa_\alpha}_{L^2(\R^d)} &=
    \smallnorm{k_\alpha}_{L^2(\R^d)}
    = \alpha^{-n/2} \smallnorm{k}_{L^2(\R^d)},
    \\
    \smallnorm{\kappa_\alpha}_{L^\infty(\R^d)} &\leq
    \smallnorm{k_\alpha}_{L^1(\R^d)}
    = \smallnorm{k}_{L^1(\R^d)},
    \\
    \smallnorm{h_\beta}_{L^2(\R^d)}
    &= \smallnorm{\fhat{h}_\beta}_{L^2(\R^d)}
    = \beta^{-n/2} \smallnorm{h}_{L^2(\R^d)},
    \\
    \smallnorm{h_\beta}_{L^\infty(\R^d)}
    &\leq \smallnorm{\fhat{h}_\beta}_{L^1(\R^d)}
    = \smallnorm{h}_{L^1(\R^d)}.
\end{align*}
\end{remark}

We now turn to uniform approximation of \(Bf\) by \(B_{\alpha,\beta}f\).
This is not meaningful unless \(Bf\in L^\infty(\R^d)\), 
and therefore we may only consider \(f\) coming from some subspace of \(L^2(\R^d)\) with this property.
We let it suffice to study functions coming from the space
\begin{equation*}
    \fourier^{-1}[L^1(\R^d)]\cap L^2(\R^d)
    \subseteq B^{-1}[L^\infty(\R^d)],
\end{equation*}
on which \(B_{\alpha,\beta}\) does indeed approximate \(B\).

\begin{theorem}[Pointwise approximation in \(L^\infty(\R^d)\)]
\label{thm:regfouriermult_pwuniformapprox}
    Define the families \((k_\alpha)_{\alpha>0}\) and \((h_\beta)_{\beta>0}\) as in Theorem \ref{thm:regfouriermult_pwl2approx}.
    Then, for any \(f\in L^2(\R^d)\) with \(\fhat{f}\in L^1(\R^d)\) we have
    \begin{equation*}
        \norm{B_{\alpha,\beta}f - Bf}_{L^\infty(\R^d)}\xrightarrow[\alpha,\beta\to 0]{} 0.
    \end{equation*}
\end{theorem}
\begin{proof}
    First note that \(\kappa = \fhat{k}\) is a continuous function by the Riemann-Lebesgue lemma, and that
    \begin{equation*}
        \kappa_\alpha(\xi)
        = \kappa (\alpha\xi)
        \xrightarrow[\alpha\to 0]{}
        \kappa(0)
        = \int k(x)\di x = 1,
    \end{equation*}
    and hence
    \begin{equation*}
        \kappa_\alpha(\xi)\fhat{f}(\xi)
        \xrightarrow[\alpha\to 0]{} \fhat{f}(\xi)
    \end{equation*}
    for almost every \(\xi\in\R^d\).
    Moreover, we have
    \begin{equation*}
        \abs{\kappa_\alpha(\xi)\fhat{f}(\xi)}
        \leq
        \norm{\kappa_\alpha}_{L^\infty(\R^d)}
        \abs{\fhat{f}(\xi)}
        =
        \norm{\kappa}_{L^\infty(\R^d)}
        \abs{\fhat{f}(\xi)}
    \end{equation*}
    for all \(\xi\),
    and since \(\fhat{f}\) is integrable,
    we have
    \begin{equation*}
        \smallnorm{\kappa_\alpha \fhat{f} - \fhat{f}}_{L^1(\R^d)}
        \xrightarrow[\alpha\to 0]{}0
    \end{equation*}
    by the dominated convergence theorem.
    
    Assuming \(\smallnorm{\mu}_{L^\infty(\R^d)}=1\) for simplicity,
    we get:
    \begin{align*}
        \smallnorm{B_{\alpha,\beta}f - Bf}_{L^\infty(\R^d)}
        &=
        \smallnorm{
        \fourier^{-1}
        \left[
        \mu \kappa_\alpha \fourier (h_\beta f)
        -
        \mu \fourier f
        \right]
        }_{L^\infty(\R^d)}\\
        &\leq
        \smallnorm{\mu}_{L^\infty(\R^d)}
        \smallnorm{
        \kappa_\alpha \fourier (h_\beta f)
        -
        \fhat{f}  
        }_{L^1(\R^d)}\\
        &=
        \smallnorm{
        \kappa_\alpha \left( \fhat{h}_\beta *\fhat{f}\right)
        -
        \fhat{f}
        }_{L^1(\R^d)}\\
        &\leq
        \smallnorm{
        \kappa_\alpha \left( \fhat{h}_\beta *\fhat{f}\right)
        -
        \kappa_\alpha \fhat{f}
        }_{L^1(\R^d)}
        +
        \smallnorm{
        \kappa_\alpha \fhat{f}
        -
        \fhat{f}
        }_{L^1(\R^d)}\\
        &\leq
        \smallnorm{\kappa_\alpha}_{L^\infty(\R^d)}
        \smallnorm{
        \fhat{h}_\beta *\fhat{f}
        -
        \fhat{f}
        }_{L^1(\R^d)}
        +
        \smallnorm{
        \kappa_\alpha \fhat{f}
        -
        \fhat{f}
        }_{L^1(\R^d)}\\
        &\leq
        \smallnorm{\kappa}_{L^\infty(\R^d)}
        \smallnorm{
        \fhat{h}_\beta *\fhat{f}
        -
        \fhat{f}
        }_{L^1(\R^d)}
        +
        \smallnorm{
        \kappa_\alpha \fhat{f}
        -
        \fhat{f}
        }_{L^1(\R^d)}
        \xrightarrow[\alpha,\beta\to 0]{}0.
    \end{align*}
    Here we have used the fact that \((\fhat{h}_\beta)_{\beta>0}\) is an approximation to the identity.
    This concludes the proof.
\end{proof}

\subsection{The wave equation: a numerical example}
Having proved Theorems \ref{thm:regfouriermult_pwl2approx} and \ref{thm:regfouriermult_pwuniformapprox},
we can confidently employ \(B_{\alpha,\beta}\) to tackle the adversarial sequences for wave propagation defined in Section \ref{sec:weqzeroinit}.
We make use of the decomposition \(B_{\alpha,\beta}=BT_{\alpha,\beta}\),
and choose simple radial filters,
\begin{equation*}
    h(x) = k(x) = \max\left\{1-\enorm{x},0\right\},
\end{equation*}
which are shown in Fig.~\ref{fig:weq-kernels} along with their dilations.
Figure \ref{fig:weq-regularized} shows the evolution of adversarially perturbed radial waves, as well as approximations by \(T_{\alpha,\beta}\) and \(B_{\alpha,\beta}\).

\begin{figure}
    \centering
    \includegraphics[width=\textwidth]{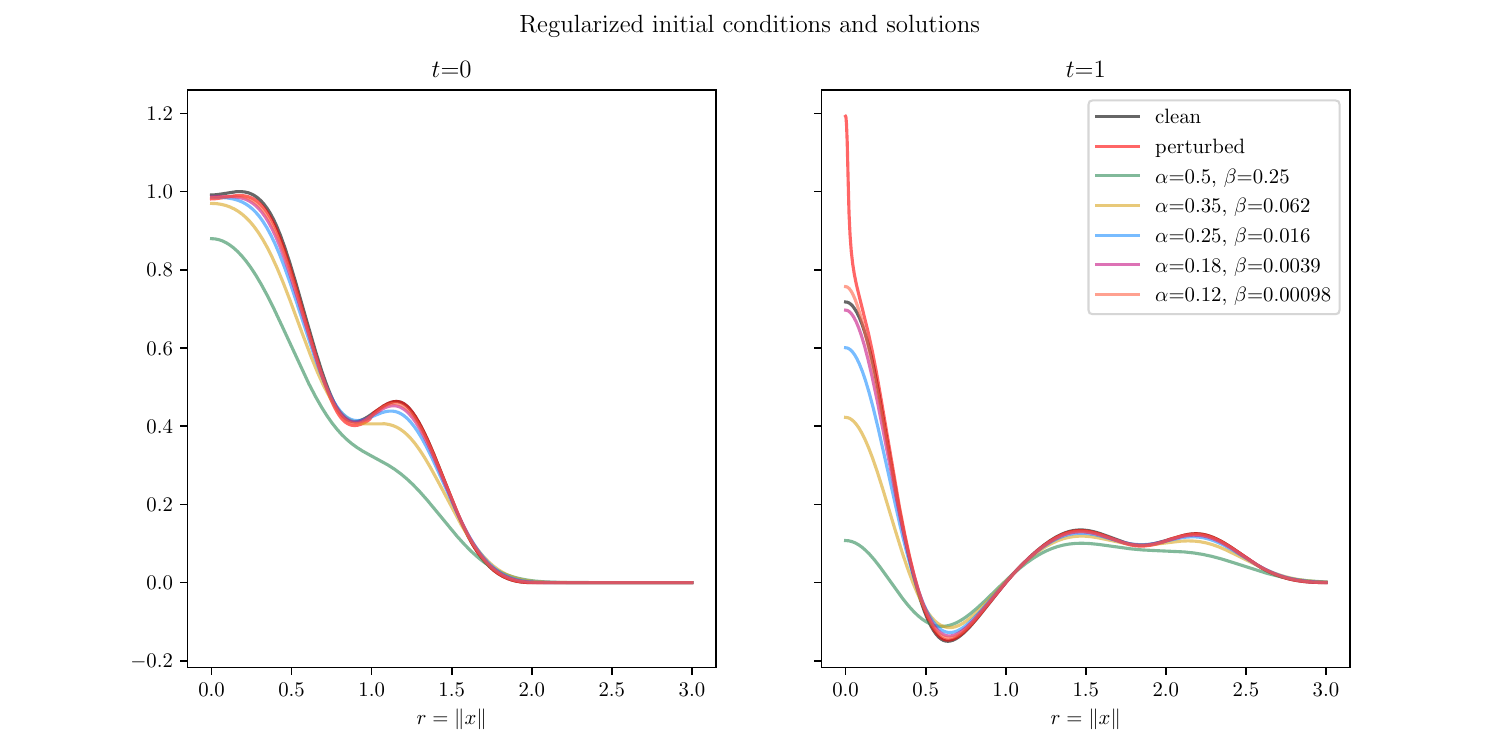}
    \caption{
    Radial component of spherical waves at times \(t=0\) and \(t=1\).
    \textbf{Left:} the clean initial state \(f\), a perturbed initial state \(f+r\), and regularized initial states \(T_{\alpha,\beta}(f+r)\).
    The perturbation satisfies \(\norm{r}_{L^\infty(\R^3)}=0.01\cdot\norm{f}_{L^\infty(\R^3)}\). The initial state and the perturbation are the same as in Fig.~\ref{fig:weq-signalpert}.
    \textbf{Right:} the end states \(Bf\), \(B(f+r)\), and \(B_{\alpha,\beta}(f+r)=BT_{\alpha,\beta}(f+r)\).}
    \label{fig:weq-regularized}
\end{figure}

\begin{figure}
    \centering
    \includegraphics[width=\textwidth]{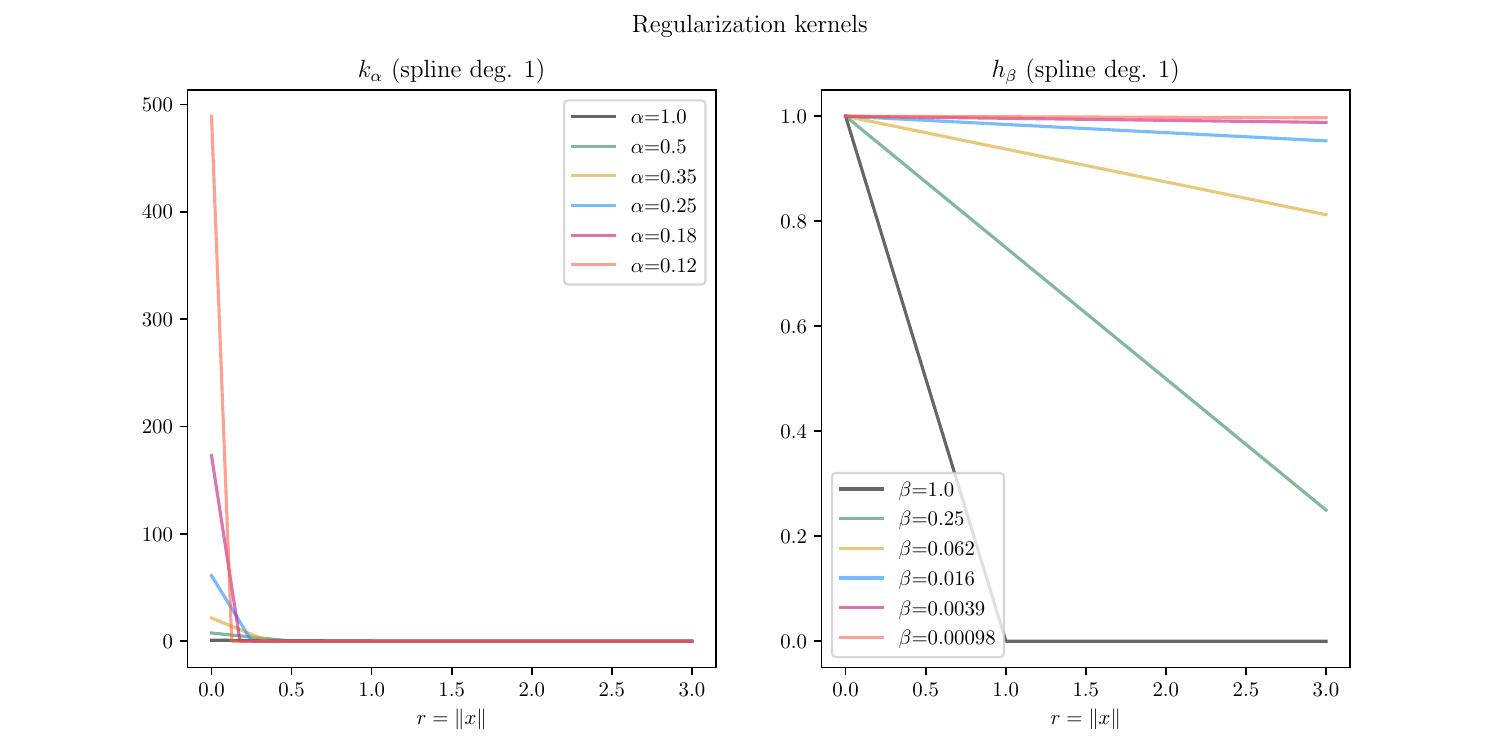}
    \caption{The regularization filters, \(k_\alpha\) and \(h_\beta\), used to define the preprocessing operator \(T_{\alpha,\beta}\colon f \mapsto k_\alpha*(h_\beta f)\) used in Fig.~\ref{fig:weq-regularized}.}
    \label{fig:weq-kernels}
\end{figure}

\section{Uniform approximation with spectral filtering and applications to inverse problems}\label{sec:IP}

In this section, \(A\colon X\to Y\) will denote a compact  linear operator between two separable Hilbert spaces with an infinite dimensional range.
The space \(Y\) is otherwise arbitrary, but we will identify \(X\) with \(L^2(\mathcal{M})\) for some measure space \(\mathcal{M}\).
We are interested in finding a regularization of the pseudoinverse \(\pseudoinv{A}\colon Y\to X\) that is also well behaved as an operator \(Y\to X'\) with the auxiliary space being \(X'=L^\infty(\mathcal{M})\).

As a compact operator, \(A\) has a singular value decomposition
\begin{equation*}
    Ax = \sum_{n=1}^\infty \sigma_n\langle x, v_n\rangle_X u_n,\quad x\in X,
\end{equation*}
where \((u_n)_{n\in\N}\) and \((v_n)_{n\in\N}\) are orthonormal systems in \(Y\) and \(X\), respectively, and \((\sigma_n)_{n\in\N}\) is a non-increasing sequence of positive real numbers with $\sigma_n\to 0$ as $n\to+\infty$.
Likewise, the pseudoinverse \(\pseudoinv{A}\) can be expressed by
\begin{equation*}
    \pseudoinv{A}y = \sum_{n=1}^\infty \frac{\langle y, u_n\rangle_Y}{\sigma_n} v_n,
\end{equation*}
for
\begin{equation*}
y\in\domain{\pseudoinv{A}}
=
\left\{
y\in Y\colon
\sum_{n=1}^\infty \sigma_n^{-2}\abs{\langle y,u_n\rangle_Y}^2 < \infty
\right\}\subsetneq Y.
\end{equation*}
We make use of this fact heavily in what follows.

\subsection{Instability in Tikhonov regularization}
Consider the Tikhonov regularizer \cite{engl2000regularization}
\begin{equation*}
    T^{\mathrm{Tikh}}_{\alpha}\colon Y\to X, \quad \alpha>0,
\end{equation*}
which can be expressed in terms of the SVD of \(A\) as
\begin{equation*}
    T^{\mathrm{Tikh}}_{\alpha}y = 
    \sum_{n=1}^\infty \frac{\sigma_n}{\sigma_n^2+\alpha}\langle y, u_n\rangle_Y v_n,\quad y\in Y.
\end{equation*}
This is a bounded operator as a mapping \(Y\to X\), with \(X=L^2(\mathcal{M})\),
but it is not, in general, bounded as a map \(Y\to X'\) for \(X'=L^\infty(\mathcal{M})\).

In the simplest case, the left singular vectors \((u_n)_{n\in\N}\) form an adversarial sequence for \(T^{\mathrm{Tikh}}_{\alpha}\) (cf. Definition \ref{def:weq:aps}).
While they are bounded in \(Y\) (with norms equal to 1), we have
\begin{align*}
    \smallnorm{T^{\mathrm{Tikh}}_{\alpha} u_n}_{L^\infty(\mathcal{M})}
    &=
    \frac{\sigma_n}{\sigma_n^2+\alpha}\smallnorm{v_n}_{L^\infty(\mathcal{M})}
    \\
    &\geq 
    \frac{1}{\sigma_1^2+\alpha}
    \sigma_n\smallnorm{v_n}_{L^\infty(\mathcal{M})},
    \quad\text{and}\\
    \smallnorm{T^{\mathrm{Tikh}}_{\alpha} u_n}_{L^\infty(\mathcal{M})}
    &\leq
    \frac{1}{\alpha}
    \sigma_n\smallnorm{v_n}_{L^\infty(\mathcal{M})}.
\end{align*}
Therefore, for every $\alpha>0$, the sequence $(\smallnorm{T^{\mathrm{Tikh}}_{\alpha} u_n}_{L^\infty(\mathcal{M})})_n$ is bounded if and only if the sequence \((\sigma_n\smallnorm{v_n}_{L^\infty(\mathcal{M})})_n\) is bounded.
In particular, if \((\sigma_n\smallnorm{v_n}_{L^\infty(\mathcal{M})})_n\notin\ell^\infty\), then  \(T^{\mathrm{Tikh}}_{\alpha}\) is not bounded from $Y$ into \(L^\infty(\mathcal{M})\).
Further, if one of the right singular vectors \((v_n)_{n\in\N}\) is itself an unbounded function, then \(T^{\mathrm{Tikh}}_{\alpha}\) is not bounded from $Y$ into $X'$.

This is only a necessary condition: there are situations where the sequence \((\sigma_n\smallnorm{v_n}_{L^\infty(\mathcal{M})})\) is bounded, but the Tikhonov regularizer is not bounded as an operator into \(L^\infty(\mathcal{M})\).
To see this, we consider the following example.
\begin{example}[Periodic convolution]
\label{ex:singularconvolution}
Let \(X=Y=L^2(\T)\), where \(\T\) is the unit circle (identified with \([0,1]\)), and
consider the periodic convolution operator \(A\colon L^2(\T)\to L^2(\T),\)
defined for \(x\in L^2(\T)\) by
\begin{equation*}
    \left(Ax\right)(t) = \left(k*x\right)(t)
    =
    \int_0^1 k(t-s)x(s)\di s,\quad t\in\T,
\end{equation*}
with some kernel \(k\in L^1(\T)\).
This operator is diagonal in the Fourier basis
\(
(v_n)_{n\in\Z}=
(\expon{2\pi\imag n\cdot})_{n\in\Z}\)
and its singular values are the absolute values of the non-zero Fourier coefficients of \(k\).
Assuming, for ease of notation, that all Fourier coefficients of \(k\) are non-zero, we thus obtain an SVD of \(A\) with 
\((v_n)_{n\in\Z}\) as right singular vectors, and a phase shifted version \((u_n)_{n\in\Z}\) of the same basis as left singular vectors.

Suppose that the kernel is badly singular, so that \(k\in L^1(\T)\setminus L^2(\T)\),
and that the (rearranged) singular values satisfy
\begin{equation*}
    \sigma_n=\abs{\fhat{k}[n]}
    \geq \frac{1}{\left(\abs{n}+1\right)^\rho},
    \quad\text{for all}\quad n\in\Z
\end{equation*}
with some \(0<\rho <1/2\).
Define perturbations \((r_N)_{N\in\N}\) by
\begin{equation*}
    r_N = \sum_{n=-N}^N a_n u_n\in Y,
\qquad
    a_n = \frac{1}{\left(\abs{n}+1\right)^{1-\rho}}.
\end{equation*}
Since \(1-\rho>1/2\), we have \((a_n)_{n\in\Z}\in\ell^2(\Z)\), and hence the perturbations are bounded by
\(\smallnorm{r_N}_Y^2 \leq \sum_{n\in \Z} a_n^2 < \infty\).

Adversarial artifacts now appear at \(0\):
\begin{align*}
    \left(T^{\mathrm{Tikh}}_{\alpha} r_N \right)(0)
    &= \sum_{n=-N}^N \frac{\sigma_n }{\sigma_n^2+\alpha}a_n v_n(0)\\
    &\geq
    \frac{1}{\sigma_\mathrm{max}^2+\alpha}
    \sum_{n=-N}^N
    a_n\sigma_n v_n(0)\\
    &\geq
    \frac{1}{\sigma_\mathrm{max}^2+\alpha}
    \sum_{n=-N}^N
    \frac{1}{\left(\abs{n}+1\right)^{1-\rho}}\frac{1}{\left(\abs{n}+1\right)^{\rho}}\cdot 1\\
    &=
    \frac{1}{\sigma_\mathrm{max}^2+\alpha}
    \sum_{n=-N}^N
    \frac{1}{\abs{n}+1}
    \xrightarrow[N\to\infty]{}\infty,
\end{align*}
where \(\sigma_\mathrm{max}\geq \sigma_n\) for all \(n\in\Z\).
The perturbations \((r_N)_{N\in\N}\) form an adversarial sequence, and \(r = \lim_{N\to\infty} r_N\in Y\) is an adversarial perturbation, namely, $T^{\mathrm{Tikh}}_{\alpha}(Y)\not\subseteq L^\infty(\mathcal{M})$.
Hence,
\(T^{\mathrm{Tikh}}_{\alpha}\) is not bounded \(Y\to L^\infty(\mathcal{M})\).
\end{example}

\subsection{Spectral filtering}
In order to ensure a bounded reconstruction into \(X'=L^\infty(\mathcal{M})\),
we modify the Tikhonov regularizer by placing non-uniform weights in the SVD.
That way, we can regularize more strongly in those directions that correspond to problematic right singular vectors.

Let \(c=(c_n)_{n\in\N}\) be a sequence of real numbers
that is bounded below by some positive number \(c_0\),
\begin{equation*}
    0 < c_0 \leq c_n\quad\text{for all}\ n\in\N.
\end{equation*}
Denote by \(T^c_\alpha\) the \emph{spectral filtering} operator
\begin{equation}
\label{eq:spectralfilterdef}
    T^c_\alpha y = \sum_{n=1}^\infty
    \frac{\sigma_n}{\sigma_n^2 + \alpha c_n}
    \langle y,u_n \rangle_Y
    v_n,\quad y\in Y.
\end{equation}
The threshold \(c_0\) ensures that this operator is dominated by a Tikhonov regularizer,
\begin{equation*}
    \smallnorm{T^c_\alpha y}_X^2
    =
    \sum_{n=1}^\infty
    \frac{\sigma_n^2\abs{\langle y,u_n \rangle_Y}^2}{\left(\sigma_n^2 + \alpha c_n\right)^2}
    \leq
    \sum_{n=1}^\infty
    \frac{\sigma_n^2\abs{\langle y,u_n \rangle_Y}^2}{\left(\sigma_n^2 + \alpha c_0\right)^2}
    =
    \smallnorm{T^{\mathrm{Tikh}}_{\alpha c_0} y}_X^2
\end{equation*}
for all \(y\in Y\),
and is thus bounded with
\begin{equation}
\label{eq:spectralbounded}
    \smallnorm{T^c_\alpha}_{Y\to X}
    \leq
    \frac{1}{\sqrt{\alpha c_0}}
\end{equation}
by Lemma \ref{thm:tikhonovbounded}.
The fact that \(\smallnorm{T^c_\alpha}_{Y\to X}\leq \smallnorm{T^\mathrm{Tikh}_{\alpha c_0}}_{Y\to X}\) will later be useful for obtaining a convergence rate for \(T^c_\alpha y\xrightarrow[\alpha\to 0]{}\pseudoinv{A}y\).

We remark that an alternative formulation of this operator is
\begin{equation*}
    T^c_\alpha y =
    \argmin_{x\in \domain{W}}
    \norm{Ax-y}_Y^2
    +
    \alpha
    \norm{Wx}_X^2,
\end{equation*}
where \(W\colon \domain{W}\to X\) is the linear operator
\begin{equation*}
    Wx = \sum_{n=1}^\infty \sqrt{c_n} \langle x, v_n\rangle_X v_n,
\end{equation*}
defined on a subspace \(\domain{W}\subseteq X\).
With \(c_n=1\) for all \(n\in\N\), we recover the familiar formulation of Tikhonov regularization.

Next, we show that spectral filtering is indeed a regularization method in \(X=L^2(\mathcal{M})\), 
and that, under mild conditions, it converges in \(X'=L^\infty(\mathcal{M})\) as well. Note that in order for it to be a regularization as an operator from $Y$ to $X'$, we also need to establish that $T_\alpha^c:Y \to X'$ is bounded. This is the subject of Section \ref{Sec:boundedness}.

\subsubsection{Convergence}
We now prove two  approximation results: pointwise in $X$ for all $y \in \mathcal{D}(A^\dagger)$ and pointwise in $X'$ under more restrictive assumptions. 
These theorems do not provide a convergence rate, as that will require more restrictions on \(y\).

\begin{theorem}[Pointwise approximation in \(X\)]
\label{thm:spectralfilter_pwl2approx}
    We have
    \begin{equation*}
        \smallnorm{\pseudoinv{A} y - T^c_\alpha y}_X
        \xrightarrow[\alpha\to 0]{}0
    \end{equation*}
    for all \(y\in \domain{\pseudoinv{A}}\).
\end{theorem}

\begin{proof}
Let \(\varepsilon>0\).
We will show that the error is smaller than \(\varepsilon\) for all \(\alpha\) smaller than some \(\alpha_0>0\).

Since \(y\in\domain{\pseudoinv{A}}\) and by Picard's criterion,
we can select an integer \(N\) large enough so that
\begin{equation*}
    \sum_{n=N+1}^\infty \frac{\abs{\langle y,u_n\rangle_Y}^2}{\sigma_n^2} <  \frac{\varepsilon^2}{4}
\end{equation*}
and the error can be bounded as follows:
\begin{align*}
    \smallnorm{\pseudoinv{A} y - T^c_\alpha y}_X
    &=
    \norm{
    \sum_{n=1}^\infty \left( \frac{1}{\sigma_n} - \frac{\sigma_n}{\sigma_n^2+\alpha c_n} \right) \langle y, u_n\rangle_Y v_n
    }_X
       \\&\leq
    \sqrt{
    \sum_{n=1}^N 
    \frac{\alpha^2c_n^2 \abs{\langle y, u_n\rangle_Y}^2}{\sigma_n^2\left(\sigma_n^2+\alpha c_n\right)^2}
    }
    +
    \sqrt{\sum_{n=N+1}^\infty \frac{\alpha^2 c_n^2 \abs{\langle y,u_n\rangle_Y}^2}{\sigma_n^2 \left(\sigma_n^2+\alpha c_n\right)^2}}
    \\&<
    \alpha
    \sqrt{
    \sum_{n=1}^N 
    \frac{c_n^2 \abs{\langle y, u_n\rangle_Y}^2}{\sigma_n^6}
    }    
    +\frac{\varepsilon}{2}=:
    \alpha
    C_N
    +\frac{\varepsilon}{2}.
\end{align*}
The result now follows by choosing \(\alpha_0\) small enough so that \(\alpha_0 C_N < \varepsilon/2\). 
\end{proof}

Note that $C_N$, and hence $\alpha_0,$ depends on $y$, $c_1,\dots,c_N$ and the SVD of $A$. As previously stated, we are interested in finding solutions to \(Ax=y\) that lie in \(X'=L^\infty(\mathcal{M})\).
Therefore, our regularization method should approximate \(\pseudoinv{A}y\) uniformly whenever \(\norm{\pseudoinv{A}y}_{L^\infty(\mathcal{M})}<\infty\).
Write
\begin{equation}
\label{eq:uniformsourceset}
    U
    =
    \left\{
    y\in Y\colon 
    \sum_{n=1}^\infty \frac{\abs{\langle y,u_n\rangle_Y}}{\sigma_n}\smallnorm{v_n}_{L^\infty(\mathcal{M})} < \infty
    \right\}.
\end{equation}
This is a subspace of \(Y\), and
\(y\in U\) is
a somewhat stronger condition than \(\smallnorm{\pseudoinv{A}y}_{L^\infty(\mathcal{M})}<\infty\), indeed
\[
\smallnorm{\pseudoinv{A}y}_{L^\infty(\mathcal{M})} =\norm{
    \sum_{n=1}^\infty  \frac{\langle y, u_n\rangle_Y}{\sigma_n}    v_n
    }_{L^\infty(\mathcal{M})}\le \sum_{n=1}^\infty \frac{\abs{\langle y,u_n\rangle_Y}}{\sigma_n}\smallnorm{v_n}_{L^\infty(\mathcal{M})}.
\]
Similarly, if $y\in U$, then $T^c_\alpha y\in L^\infty(\mathcal{M})$.
We have the following convergence result, whose proof is similar to that of Theorem~\ref{thm:spectralfilter_pwl2approx}.

\begin{theorem}[Pointwise approximation in \(L^\infty(\mathcal{M})\)]
\label{thm:spectralfilter_pwuniformapprox}
    We have
    \begin{equation*}
        \smallnorm{\pseudoinv{A} y - T^c_\alpha y}_{L^\infty(\mathcal{M})}
        \xrightarrow[\alpha\to 0]{}0
    \end{equation*}
    for all \(y\in U\).
\end{theorem}

\begin{proof}
As before, we need to show that for any  \(\varepsilon>0\) there  exists some \(\alpha_0>0\), such that for all $\alpha < \alpha_0,$
 $$
 \smallnorm{\pseudoinv{A} y - T^c_\alpha y}_{L^\infty(\mathcal{M})} < \varepsilon.
 $$

Since \(y\in U\),
we can select an integer \(N\) large enough so that
\begin{equation*}
    \sum_{n=N+1}^\infty \frac{\abs{\langle y,u_n\rangle_Y}}{\sigma_n}\smallnorm{v_n}_{L^\infty(\mathcal{M})} < 
    \frac{\varepsilon}{2}.
\end{equation*}
The error can now be bounded as follows:
\begin{align*}
    \smallnorm{\pseudoinv{A} y - T^c_\alpha y}_{L^\infty(\mathcal{M})}
    & =
    \norm{
    \sum_{n=1}^\infty \left( \frac{1}{\sigma_n} - \frac{\sigma_n}{\sigma_n^2+\alpha c_n} \right) \langle y, u_n\rangle_Y v_n
    }_{L^\infty(\mathcal{M})}
    \\&\leq
    \norm{
    \sum_{n=1}^N 
    \frac{\alpha c_n
    \langle y, u_n\rangle_Y}{\sigma_n\left(\sigma_n^2+\alpha c_n\right)} v_n
    }_{L^\infty(\mathcal{M})}+\norm{
    \sum_{n=N+1}^\infty 
    \frac{\alpha c_n
    \langle y, u_n\rangle_Y}{\sigma_n\left(\sigma_n^2+\alpha c_n\right)} v_n
    }_{L^\infty(\mathcal{M})}
    \\&<
    \alpha
    \sum_{n=1}^N 
    \frac{c_n
    \abs{\langle y, u_n\rangle_Y}}{\sigma_n^3}
    \smallnorm{v_n
    }_{L^\infty(\mathcal{M})}
    +
    \frac{\varepsilon}{2}
    =:
    \alpha C_N + \frac{\varepsilon}{2}.
\end{align*}
Choosing \(\alpha_0\) so that \(\alpha_0 C_N < \varepsilon/2\) yields the result.
\end{proof}

\subsubsection{Convergence rates on source sets}
Theorems \ref{thm:spectralfilter_pwl2approx} and \ref{thm:spectralfilter_pwuniformapprox} provide no rate of convergence in terms of \(\alpha\).
In order to establish a convergence rate,
we require some degree of ``smoothness'' of \(y\), something commonly known as a ``source condition'' in inverse problems \cite{engl2000regularization}.
Recall that the operator \(A\) is assumed to be compact, so that \(\sigma_n \to 0\) as $n \to \infty$.

\begin{definition}
\label{def:strongpicard}
    Let \(\eta\geq 0\). 
    We say that \(y\in Y\) satisfies the \emph{\(\eta\)-strong Picard criterion}
    if
    \begin{equation}
    \label{eq:picardstrong}
        P_\eta(y) := \sum_{n=1}^\infty
        \frac{\abs{\langle y, u_n \rangle_Y}^2}{
        \sigma_n^{2+\eta}
        }
        <
        \infty
    .\end{equation}
\end{definition}

\begin{lemma}
\label{thm:strongpicardnested}
    Suppose that \(A\) is a compact operator,
    and let \(\eta\geq 0\).
    The set of elements of \(Y\) that satisfy the \(\eta\)-strong Picard criterion,
    \begin{equation*}
        V_\eta = \left\{ y\in Y\colon P_\eta(y) < \infty \right\}
    \end{equation*}
    forms a subspace of \(Y\),
    and for any \(\mu\leq \eta\), we have
    \begin{equation*}
    V_\eta\subseteq V_\mu\subseteq V_0=\domain{\pseudoinv{A}}.
    \end{equation*}
\end{lemma}
\begin{proof}

For \(y,z\in Y\), and \(\beta,\gamma\in\C\), we have \
\begin{equation*}
P_\eta(\beta y+\gamma z)\leq 2\abs{\beta}^2P_\eta(y)+2\abs{\gamma}^2P_\eta(z) < \infty
,\end{equation*}
where we have used the triangle inequality and the basic fact that \((a+b)^2\leq 2a^2+2b^2\) for any \(a,b\in\R\). 
This proves that \(V_\eta\) is a subspace of \(Y\).

Now let \(0\leq \mu\leq \eta\).
Since \(A\) is compact, we can find an \(N\in\N\) such that \(\sigma_n\leq 1\) for all \(n\geq N\).
Then \(\sigma_n^{-1} \geq 1\), so \(\sigma_n^{-\mu} \leq \sigma_n^{-\eta}\) for \(n\geq N\),
and therefore
\begin{align*}
    P_\mu(y)
     & \leq
        \sum_{n=1}^N
        \sigma_n^{-\mu}\frac{\abs{\langle y, u_n \rangle_Y}^2}{
        \sigma_n^{2}
        }
        +
        \sum_{n=N+1}^\infty
        \sigma_n^{-\eta}\frac{\abs{\langle y, u_n \rangle_Y}^2}{
        \sigma_n^{2}
        }
    \\ & \leq
        \sum_{n=1}^N
        \sigma_n^{-\mu}\frac{\abs{\langle y, u_n \rangle_Y}^2}{
        \sigma_n^{2}
        }
        +
        P_\eta(y)
.\end{align*}
In other words, \(P_\eta(y)<\infty\) implies \(P_\mu(y)<\infty\),
and therefore \(V_\eta\subseteq V_\mu\).
\end{proof}

\begin{remark}
    For \(\eta>0\),
    inclusion in the space \(V_\eta\) can also be characterized in terms of
    \emph{source sets} in \(X\).
    One can show that
    \begin{equation*}
        y\in V_\eta
        \quad\text{if and only if}\quad 
        \pseudoinv{A}y \in X_{\eta/4}=\range{\left(A^*A\right)^{\eta/4}}
    \end{equation*}
    (see Proposition 3.13 of \cite{engl2000regularization}).
    However, the explicit dependence of \(V_\eta\) on the singular values of \(A\) through \eqref{eq:picardstrong} is more useful to our analysis, so we state our results in terms of these subspaces of \(Y\).
\end{remark}

To establish rates of convergence for \(T^c_\alpha\) on the source sets,
we follow a similar approach as in Chapter 4 of \cite{engl2000regularization}.
However, it needs to be adapted to the current case where the spectral filter does not depend exclusively on the singular values \(\sigma_n\), but rather on their index \(n\).
By imposing a condition on the growth of the coefficients \(c_n\) relative to the decay of \(\sigma_n\), standard proofs can be emulated.
This comes at the cost of having to assume more smoothness of the data.
Before we state the next convergence results, we need the following calculus fact.

\begin{lemma}
    \label{thm:spectralfilter_ratebound}
    Let \(\eta,\alpha>0\) and define the function \(h\colon [0,\infty)\to \R\) by
    \begin{equation*}
        h(\sigma) = \frac{\sigma^\eta}{\left(\sigma^2 + \alpha\right)^2},\qquad \sigma>0.
    \end{equation*}
    If \(\eta < 4\), then
    \begin{equation*}
        h(\sigma) \leq C_\eta \alpha^{\eta/2-2}
    \end{equation*}
    for some \(C_\eta>0\),
    and if \(\eta\geq 4\), then \(h\) is strictly increasing.
\end{lemma}
\begin{proof}
See \ref{sec:appendix_proofs}.
\end{proof}

\begin{theorem}[Pointwise approximation in \(X\) with convergence rate]
\label{thm:spectralfilter_pwl2approx_rate}
    Suppose that
    \(A\) is compact,
    and that the weights \((c_n)_{n\in\N}\) satisfy
    \begin{equation}
        \label{eq:spectralfilter_pwl2approx_rate_condition}
        c_n \leq C \sigma_n^{-\beta}\quad\text{for all}\ n\in\N,
    \end{equation}
    for some \(C, \beta>0\).
    Let \(y\in V_{\eta+2\beta}\) for some \(\eta>0\).
    Then there exists a constant \(C'>0\), depending only on $C$, $\beta$, $\eta$, $\sigma_1$ and $y$,
    such that
    \begin{equation*}
        \smallnorm{\pseudoinv{A} y - T^c_\alpha y}_{L^2(\mathcal{M})}
        \leq 
        C' \alpha^{\min\{1,\eta/4\}}
    \end{equation*}
    for all \(\alpha>0\).
\end{theorem}

\begin{proof}
Taking squared norms, we obtain 
\begin{align*}
    \smallnorm{\pseudoinv{A} y - T^c_\alpha y}_{L^2(\mathcal{M})}^2
    & = \sum_{n=1}^\infty \frac{(\alpha c_n)^2 \abs{\langle y, u_n\rangle_Y}^2}{\sigma_n^2\left(\sigma_n^2+\alpha c_n\right)^2}
    \\ & =
    \sum_{n=1}^\infty \frac{ \abs{\langle y, u_n\rangle_Y}^2}{\sigma_n^{2+\eta}}(\alpha c_n)^2
    \frac{\sigma_n^\eta}{\left(\sigma_n^2+\alpha c_n\right)^2}.
\end{align*}
Now, we employ Lemma \ref{thm:spectralfilter_ratebound}.
If \(\eta < 4\), then
\begin{align*}
    \smallnorm{\pseudoinv{A} y - T^c_\alpha y}_{L^2(\mathcal{M})}^2
    &\leq
    C_\eta
    \sum_{n=1}^\infty \frac{ \abs{\langle y, u_n\rangle_Y}^2}{\sigma_n^{2+\eta}}
    (\alpha c_n)^2
    (\alpha c_n)^{\eta/2-2}\\
    &\leq
    \alpha^{\eta/2}
    C_\eta
    \sum_{n=1}^\infty \frac{ \abs{\langle y, u_n\rangle_Y}^2}{\sigma_n^{2+\eta}}
    (C\sigma_n^{-\beta})^{\eta/2}\\
    &\leq
    \alpha^{\eta/2}
    C_\eta'
    \sum_{n=1}^\infty \frac{ \abs{\langle y, u_n\rangle_Y}^2}{\sigma_n^{2+\eta+\beta\eta/2}}\\
    &=
    \alpha^{\eta/2}
    C_\eta'
    P_{\eta+\beta\eta/2}(y),
\end{align*}
and since \(\eta\leq 4\), we have \(y\in V_{\eta+2\beta}\subseteq V_{\eta+\beta\eta/2}\).
If \(\eta \geq 4\), then, by monotonicity of $h(\sigma)$, 
\begin{align*}
    \smallnorm{\pseudoinv{A} y - T^c_\alpha y}_{L^2(\mathcal{M})}^2
    &\leq
    \sum_{n=1}^\infty \frac{ \abs{\langle y, u_n\rangle_Y}^2}{\sigma_n^{2+\eta}}
    (\alpha c_n)^2
    \frac{\sigma_1^\eta}{\left(\sigma_1^2+\alpha c_n\right)^2}\\
    &\leq
    \alpha^2
    C^2
    \sigma_1^{\eta-4}
    \sum_{n=1}^\infty \frac{ \abs{\langle y, u_n\rangle_Y}^2}{\sigma_n^{2+\eta+2\beta}}=
    \alpha^2
    C^2
    \sigma_1^{\eta-4}
    P_{\eta+2\beta}(y).
\end{align*}
Taking square roots on both sides of the above inequalities gives the result.
\end{proof}

\begin{remark}
    Tikhonov regularization satisfies the conditions of Theorem~\ref{thm:spectralfilter_pwl2approx_rate} with \(\beta=0\),
    and the above rate is exactly the optimal rate for Tikhonov regularization (see  \cite[Example~4.15]{engl2000regularization}).
\end{remark}

Next, we consider uniform approximation on the sets \(V_\eta\), \(\eta>0\).
Again, we need a condition on the SVD of \(A\)
that relates
the growth of \(\smallnorm{v_n}_{L^\infty(\mathcal{M})}\) to 
the decay of \(\sigma_n\).

\begin{lemma}
\label{thm:uniformpicard}
    If the operator \(A\) satisfies
    \begin{equation}
        \label{eq:spectralrelativegrowth}
        \sum_{n=1}^\infty
        \sigma_n^{\eta} \smallnorm{v_n}_{L^\infty(\mathcal{M})}^2
        < \infty,
    \end{equation}
    then \(V_\eta\subseteq U\) (defined in Eq.~(\ref{eq:uniformsourceset})).
    Consequently, \(\pseudoinv{A}y\in L^\infty(\mathcal{M})\)
    for all \(y\in V_\eta\). 
\end{lemma}
\begin{proof}
   Take $y\in V_\eta$. We use the triangle inequality and H\"older's inequality for infinite series:
    \begin{align*}
        \smallnorm{\pseudoinv{A}y}_{L^\infty(\mathcal{M})}
        &\leq
        \sum_{n=1}^\infty
        \frac{\abs{\langle y, u_n\rangle_Y}}{\sigma_n}\norm{v_n}_{L^\infty(\mathcal{M})}
        \\
        &\leq
        \norm{
        \left(
        \frac{\abs{\langle y, u_n\rangle}}{\sigma_n^{1+\eta/2}}
        \right)_{n\in\N}
        }_{\ell^2(\N)}
        \norm{
        \left(
        \sigma_n^{\eta/2}\norm{v_n}_\infty
        \right)_{n\in\N}}_{\ell^2(\N)},\\
    \end{align*}
    which is finite by 
    \eqref{eq:spectralrelativegrowth} and using that \(y\in V_\eta\).
\end{proof}

\begin{remark}
    Condition \eqref{eq:spectralrelativegrowth}
    significantly limits the choice for the parameter \(\eta\).
    For example, if \(\norm{v_n}_{L^\infty(\mathcal{M})}=1/\sigma_n\), then \(\eta\) is always at least 2 if the condition is to be satisfied.
    Consider a mildly ill-posed problem satisfying
    \(\smallnorm{v_n}_{L^\infty(\mathcal{M})}=1/\sigma_n = n^s\) for some \(s>0\).
    Then 
    \begin{equation*}
        \sigma_n^\eta
        \smallnorm{v_n}_{L^\infty(\mathcal{M})}^2
        =
        n^{-s\eta+2s}
        =
        n^{s(2-\eta)}
    \end{equation*}
    so \eqref{eq:spectralrelativegrowth} is satisfied if and only if \(\eta>2+1/s\).
    For severely ill-posed problems with \(\smallnorm{v_n}_{L^\infty(\mathcal{M})}=1/\sigma_n = \expon{sn}\),
    any \(\eta>2\) suffices.
\end{remark}

\begin{theorem}[Pointwise approximation in \(L^\infty(\mathcal{M})\) with convergence rate]
\label{thm:spectralfilter_pwlinfapprox_rate}
    Suppose that
    \(A\) is compact and that its SVD satisfies
    \begin{equation}\label{eq:spectralrelativegrowth2}
        \sum_{n=1}^\infty
        \sigma_n^{\eta} \smallnorm{v_n}_{L^\infty(\mathcal{M})}^2
        <\infty.
    \end{equation}
    Moreover, suppose that the weights \((c_n)_{n\in\N}\) satisfy
    \begin{equation*}
        c_n \leq C \sigma_n^{-\beta}\quad\text{for all}\ n\in\N,
    \end{equation*}
    for some \(C, \beta>0\).
    Let \(y\in V_{2\eta+2\beta}\).
    Then there exists a constant \(C'>0\)
    such that
    \begin{equation*}
    \smallnorm{\pseudoinv{A} y - T^c_\alpha y}_{L^\infty(\mathcal{M})}
        \leq 
        C' \alpha^{\min\{1,\eta/4\}}
    \end{equation*}
    for all \(\alpha>0\).
\end{theorem}

\begin{proof}
We use the triangle inequality and rearrange exponents:
\begin{align*}
    \smallnorm{\pseudoinv{A} y - T^c_\alpha y}_{L^\infty(\mathcal{M})}
    & \leq
    \sum_{n=1}^\infty 
    \frac{\alpha c_n
    \abs{\langle y, u_n \rangle_Y}}{\sigma_n\left(\sigma_n^2+\alpha c_n\right)}
    \smallnorm{v_n
    }_{L^\infty(\mathcal{M})}\\
    & =
    \sum_{n=1}^\infty 
    \frac{ \abs{\langle y, u_n \rangle_Y}}{
    \sigma_n^{1+\eta}}
    (\alpha c_n)
    \frac{\sigma_n^{\eta/2}}
    {\sigma_n^2+\alpha c_n}
    \sigma_n^{\eta/2}
    \smallnorm{v_n}_{L^\infty(\mathcal{M})}.
\end{align*}
As in the proof of Theorem \ref{thm:spectralfilter_pwl2approx_rate},
we now use Lemma \ref{thm:spectralfilter_ratebound},
by remarking that since the square root is an increasing function, \(\sigma^{\eta/2}(\sigma^2+\alpha)^{-1}\) has the same maxima as \(h(\sigma)\).
For \(\eta< 4\) we get
\begin{align*}
    \smallnorm{\pseudoinv{A} y - T^c_\alpha y}_{L^\infty(\mathcal{M})}
    & \leq
    C'
    \sum_{n=1}^\infty 
    \frac{ \abs{\langle y, u_n \rangle_Y}}{
    \sigma_n^{1+\eta}}
    (\alpha c_n)
    (\alpha c_n)^{\eta/4-1}
    \sigma_n^{\eta/2}
    \smallnorm{v_n}_{L^\infty(\mathcal{M})}\\
    & \leq
    \alpha^{\eta/4}
    C'
    \sum_{n=1}^\infty 
    \frac{ \abs{\langle y, u_n \rangle_Y}}{
    \sigma_n^{1+\eta}}
    (C\sigma_n^{-\beta})^{\eta/4}
    \sigma_n^{\eta/2}
    \smallnorm{v_n}_{L^\infty(\mathcal{M})}\\
    &\leq
    \alpha^{\eta/4} C''
    \norm{\left(
    \frac{
    \abs{\langle y, u_n \rangle_Y}}{
    \sigma_n^{1+\eta+\beta\eta/4}}    
    \right)_{n\in\N}}_{\ell^2(\N)}
    \norm{\left(
    \sigma_n^{\eta/2}
    \smallnorm{v_n
    }_{L^\infty(\mathcal{M})}
    \right)_{n\in\N}}_{\ell^2(\N)}.
\end{align*}
The two $\ell^2$ norms are finite, because of (\ref{eq:spectralrelativegrowth2}) and using that \(y\in V_{2\eta+2\beta}\subseteq V_{2\eta+\beta\eta/2}\) for \(\eta\leq 4\).
For \(\eta\geq 4\):
\begin{align*}
    \smallnorm{\pseudoinv{A} y - T^c_\alpha y}_{L^\infty(\mathcal{M})}
    & \leq
    C'
    \sum_{n=1}^\infty 
    \frac{ \abs{\langle y, u_n \rangle_Y}}{
    \sigma_n^{1+\eta}}
    (\alpha c_n)
    \frac{\sigma_1^{\eta/2}}{\sigma_1^2 + \alpha c_n}
    \sigma_n^{\eta/2}
    \smallnorm{v_n}_{L^\infty(\mathcal{M})}\\
    & \leq
    \alpha
    C'
    \sigma_1^{\eta/2-2}
    \sum_{n=1}^\infty 
    \frac{ \abs{\langle y, u_n \rangle_Y}}{
    \sigma_n^{1+\eta}}
    (C\sigma_n^{-\beta})
    \sigma_n^{\eta/2}
    \smallnorm{v_n}_{L^\infty(\mathcal{M})}\\
    &\leq
    \alpha
    C''
    \sigma_1^{\eta/2-2}
    \norm{\left(
    \frac{
    \abs{\langle y, u_n \rangle_Y}}{
    \sigma_n^{1+\eta+\beta}}    
    \right)_{n\in\N}}_{\ell^2(\N)}
    \norm{\left(
    \sigma_n^{\eta/2}
    \smallnorm{v_n
    }_{L^\infty(\mathcal{M})}
    \right)_{n\in\N}}_{\ell^2(\N)},
\end{align*}
which, similarly, is finite because of (\ref{eq:spectralrelativegrowth2}) and $y \in V_{2\eta+2\beta}.$
\end{proof}

\subsubsection{Boundedness of the spectral filter}\label{Sec:boundedness}

We now have a recipe for a spectral filtering operator that guarantees convergence both in \(X=L^2(\mathcal{M})\) and in \(X'=L^\infty(\mathcal{M})\),
provided that certain conditions are satisfied. We have also noted at the beginning of this section, that this operator is bounded from $Y$ into $X$. However, for it to be indeed a regularization as an operator $T_\alpha^c: Y \to X'$, we need to establish that it is (under certain conditions) a bounded operator with respect to these spaces. 
We assume that \(A\) is compact, and that \(v_n\in L^\infty(\mathcal{M})\) for all \(n\in\N\).

As for the convergence rate results, we will again suppose that  $(\sigma_n^{\eta/2} \|v_n\|_\infty) \in \ell^2(\mathbb{N})$ for some $\eta>0$, cf.\ (\ref{eq:spectralrelativegrowth2}).
Under this condition, we can always find a sequence of weights
\((c_n)_{n\in\N}\)
such that \(T_\alpha^c: Y \to X'\) is bounded.

\begin{proposition}
\label{thm:spectralfilterbounded}
Let $A$ be a compact operator with SVD such that \(v_n\in L^\infty(\mathcal{M})\) for all \(n\in\N\) and let (\ref{eq:spectralrelativegrowth2}) hold for some $\eta>0$. If $\eta \leq 2,$ any lower bounded sequence of weights $c_n \geq c_0>0$ results in $T_\alpha^c: Y \to X'$ being bounded. If $\eta>2,$ any choice $c_n \geq C \sigma_n^{1-\eta/2},$ with arbitrary constant $C>0$ implies boundedness of  $T_\alpha^c: Y \to X'$. In both cases,
\[
\left\|T_\alpha^c  \right\|_{Y\to L^\infty(\mathcal{M})}\le \frac{C'}{\alpha},\qquad \alpha>0,
\]
for some $C'>0$ independent of $\alpha$.
\end{proposition}
\begin{proof}
If $\eta \leq 2$, for any $y \in Y,$  
\begin{align*}
\left\|T_\alpha^c y \right\|_{L^\infty(\mathcal{M})} &\leq \sum_{n=1}^\infty \frac{\sigma_n}{\sigma_n^2+\alpha c_n} \left|\langle y,u_n \rangle_Y\right| \left\|v_n\right\|_{\infty}\\
&\leq \frac{\sigma_1^{1-\eta/2}}{\alpha c_0} \sum_{n=1}^\infty \left|\langle y,u_n \rangle_Y\right| \sigma_n^{\eta/2} \left\|v_n\right\|_{\infty}\\
&\leq \frac{\sigma_1^{1-\eta/2}}{\alpha c_0}   \left\|\left( \sigma_n^{\eta/2} \left\|v_n\right\|_\infty \right)_{n \in \mathbb{N}} \right\|_{\ell^2(\mathbb{N})} \cdot \left\|y\right\|_Y.
\end{align*}
If $\eta>2,$ choosing $c_n \geq C \sigma_n^{1-\eta/2}$ yields,
$$
\frac{\sigma_n^{1-\eta/2}}{\sigma_n^2+\alpha c_n} \leq \frac{1}{\alpha C},
$$
and we can similarly deduce that for all $y \in Y,$
\begin{align*}
\left\|T_\alpha^c y \right\|_{L^\infty(\mathcal{M})} &\leq \sum_{n=1}^\infty \frac{\sigma_n}{\sigma_n^2+\alpha c_n} \left|\langle y,u_n \rangle_Y\right| \left\|v_n\right\|_{\infty} \leq \frac{1}{\alpha C} \left\|\left( \sigma_n^{\eta/2} \left\|v_n\right\|_\infty \right)_{n \in \mathbb{N}} \right\|_{\ell^2(\mathbb{N})} \cdot \left\|y\right\|_Y.
\end{align*}
\end{proof}

By Lemma~\ref{thm:uniformpicard} and Theorem~\ref{thm:spectralfilter_pwuniformapprox},
\(T^c_\alpha y\xrightarrow[\alpha\to 0]{} \pseudoinv{A}y\) in \(L^\infty(\mathcal{M})\) for all \(y\in V_\eta\),
and Theorem~\ref{thm:spectralfilter_pwlinfapprox_rate} provides a convergence rate for \(y\in V_{2\eta+2\beta} \), as long as $\beta>0$ is such that
$$
c_n \leq C \sigma_n^{-\beta}
$$
for some $C>0.$ Thus, for $\eta \leq 2,$ the choice $c_n=c_0>0$ allows for any $\beta>0$ and for $\eta >2,$ setting $c_n = C \sigma_n^{1-\eta/2}$ one can choose $\beta = -1+\eta/2.$
Therefore if \(y\in V_{2 \eta+2 \beta}\), with $\beta$ appropriately chosen depending on $\eta,$ and $r\in Y$ is the noise in the measurements with $\|r\|_Y\le\delta$, we have
\begin{align*}
    \smallnorm{T^{c}_{\alpha(\delta)} (y+r) - \pseudoinv{A}y}_{L^\infty(\mathcal{M})}
    &\leq
    \smallnorm{T^{c}_{\alpha(\delta)} y - \pseudoinv{A}y}_{L^\infty(\mathcal{M})}
    +
    \smallnorm{T^{c}_{\alpha(\delta)} r}_{L^\infty(\mathcal{M})}\\
    &\leq
    C\alpha(\delta)^{\min\{1,\eta/4\}}
    +
    \frac{C'}{\alpha(\delta)}
    \delta\\
    &=
    C'' \delta^{\min\{1/2,\eta/8\}}
\end{align*}
with the choice \(\alpha(\delta)=\sqrt{\delta}\).

\subsection{A numerical example}

We revisit Example \ref{ex:singularconvolution} to visualize the differences in approximation with 
\(T^{\mathrm{Tikh}}_\alpha\) and \(T^{c}_\alpha\).
The operator \(A\colon L^2(\T)\to L^2(\T)\) is a convolution with a singular kernel \(k\in L^1(\T)\setminus L^2(\T)\),
\begin{equation*}
    Ax (t)= \left(k*x\right)(t) = \int_\T k(s) x(t-s)\di s
\end{equation*}
for all \(t\in\T\).
We choose a real valued kernel with
\begin{equation}
\label{eq:singularkernel}
    \fhat{k}[n] = (n+1)^{-1/3}
\end{equation}
for \(n\geq 0\).
As shown in Example \ref{ex:singularconvolution} (with \(\rho=1/3\)), this leads to a Tikhonov regularizer that admits adversarial perturbations.

The SVD of \(A\) satisfies (\ref{eq:spectralrelativegrowth2})
for any \(\eta > 3\).
We choose weights for the modified regularizer 
according to Proposition~\ref{thm:spectralfilterbounded},
using \(\eta=4\),
\begin{equation*}
    c_n = \sigma_n^{1-
    \eta/2}
    =
    \sigma_n^{-1}
    =
    (\abs{n}+1)^{1/3}
\end{equation*}
for \(n\in\Z\).
By Proposition~\ref{thm:spectralfilterbounded},
the resulting operator, \(T^{c}_\alpha\)
is bounded
\(L^2(\T)\to L^\infty(\T)\)
and by Theorem \ref{thm:spectralfilter_pwuniformapprox} it is guaranteed to converge in \(L^\infty(\T)\) to \(\pseudoinv{A}\) pointwise on the subspace \(U\subseteq Y\) from \eqref{eq:uniformsourceset}.
By Theorem \ref{thm:spectralfilter_pwlinfapprox_rate},
we have
\begin{equation*}
    \smallnorm{\pseudoinv{A} y - T^c_\alpha y}_{L^\infty(\mathcal{M})}
        \leq 
        C \alpha^{\min\{1,\eta/4\}}
        =
        C \alpha
\end{equation*}
for
\(y\in V_{2\eta + 2(\eta/2-1)} = V_{10}\)
and some constant \(C>0\).

We note that 
for any 
\(s>1/2+\rho=5/6\),
the Sobolev space \(H^s(\T)\) is a subspace of \(U\).
Namely,
there exist  constants \(C,C'>0\) such that
\begin{align*}
    \sum_{n=1}^\infty \frac{\abs{\langle y,u_n\rangle_Y}}{\sigma_n}\smallnorm{v_n}_{L^\infty(\T)}
    &
    =
    \sum_{n\in\Z} \left(\abs{\fhat{y}[n]}/\abs{\fhat{k}[n]}\right)\cdot 1\\
    &
    \leq
    C
    \sum_{n\in\Z} \abs{n^\rho\ \fhat{y}[n]}\\
    &
    =
    C
    \sum_{n\in\Z} \abs{n}^{-1/2-\epsilon} \abs{n}^{1/2+\epsilon+\rho} \abs{\fhat{y}[n]}\\
    &
    \leq
    C
    \norm{\left(\abs{n}^{-1/2-\epsilon}\right)_{n\in\Z}}_{\ell^2(\Z)}
    \norm{\left(\abs{n}^{1/2+\rho+\epsilon}\abs{\fhat{y}[n]}\right)_{n\in\Z}}_{\ell^2(\Z)}\\
    &
    \leq
    C'
    \norm{y}_{H^{1/2+\rho+\epsilon}(\T)}
\end{align*}
for every \(\epsilon>0\) and every \(y\in L^2(\T)\).
Likewise, we can characterize the source sets \((V_\eta)_{\eta>0}\) as follows:
\begin{align*}
    P_\eta(y) & = \sum_{n=1}^\infty \frac{\abs{\langle y, u_n\rangle_Y}^2}{\sigma_n^{2+\eta}}\\
    &\leq
    C \sum_{n\in\Z} \abs{n}^{\rho(2+\eta)}\abs{\fhat{y}[n]}^2\\
    &= C\norm{\left(\abs{n}^{\rho+\eta\rho/2}\abs{\fhat{y}[n]}\right)_{n\in\Z}}_{\ell^2(\Z)}\\
    &
    \leq
    C
    \norm{y}_{H^{\rho + \eta\rho/2}(\T)},
\end{align*}
so
\(H^{\rho+\eta\rho/2}(\T)=H^{(2+\eta)/6}(\T)\subseteq V_\eta\) for all \(\eta>0\).

Figure \ref{fig:perconv-kernel} shows the singular kernel \(k\) along with signals \(x\in L^2(\T)\) satisfying three different smoothness criteria, and in
Figs.~\ref{fig:perconv-test0}, \ref{fig:perconv-test1}, and \ref{fig:perconv-test2},
we show regularized solutions of \(y=k*x\).
An adversarial perturbation is added to \(y\), which is constructed as explained in Example \ref{ex:singularconvolution}.
We see that Tikhonov regularization struggles to suppress the resulting spike, while the effects are milder for \(T^c_\alpha\).
In addition, we show a truncated SVD approximation of the solution, \(T^{\mathrm{TSVD}}_{\alpha}\), which can approximate smooth signals uniformly but is ill-suited for rough signals. In particular, the reconstruction in Fig.~\ref{fig:perconv-test1} fails to well-approximate the corners at the points of no differentiability, and the reconstruction in Fig.~\ref{fig:perconv-test2} suffers from the usual Gibbs phenomenon at the points of discontinuity. In both cases, the reconstructions provided by $T^c_\alpha$ are superior in quality. It is worth observing that the discontinuous function $x$ of Fig.~\ref{fig:perconv-test2} does not satisfy the assumptions of Theorem~\ref{thm:spectralfilter_pwuniformapprox}, and so uniform convergence is not guaranteed by the theory: it could serve as a motivation for extending the bounds obtained in this work to a larger class of functions.

\begin{figure}
    \centering
    \includegraphics[width=0.7\textwidth]{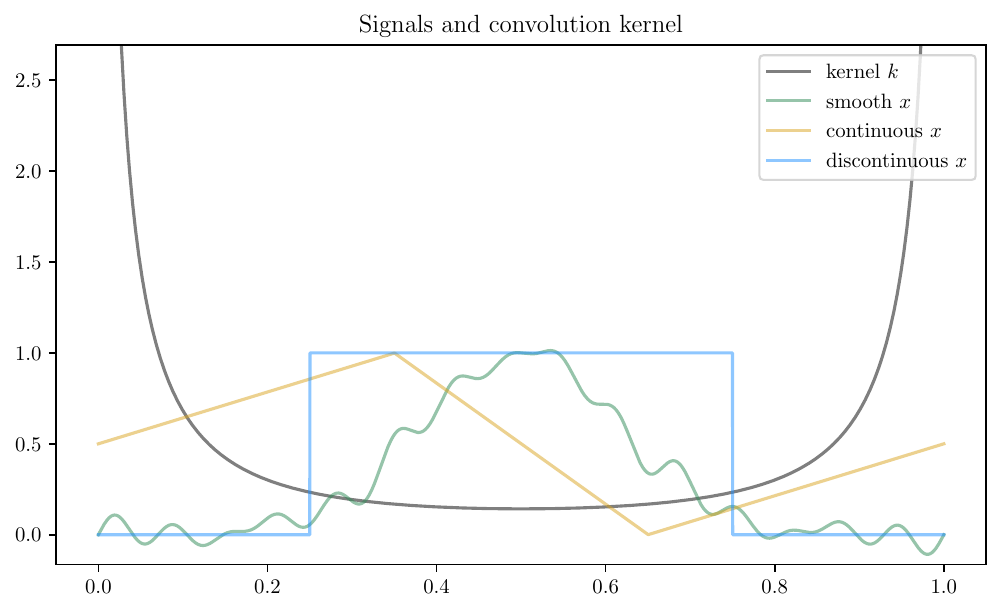}
    \caption{
    The kernel \(k\) for the forward operator \(Ax = k*x\), defined by \eqref{eq:singularkernel},
    and the three signals of Figs.~\ref{fig:perconv-test0}-\ref{fig:perconv-test2}.
    }
    \label{fig:perconv-kernel}
\end{figure}

\begin{figure}
    \centering
    \includegraphics[width=\textwidth]{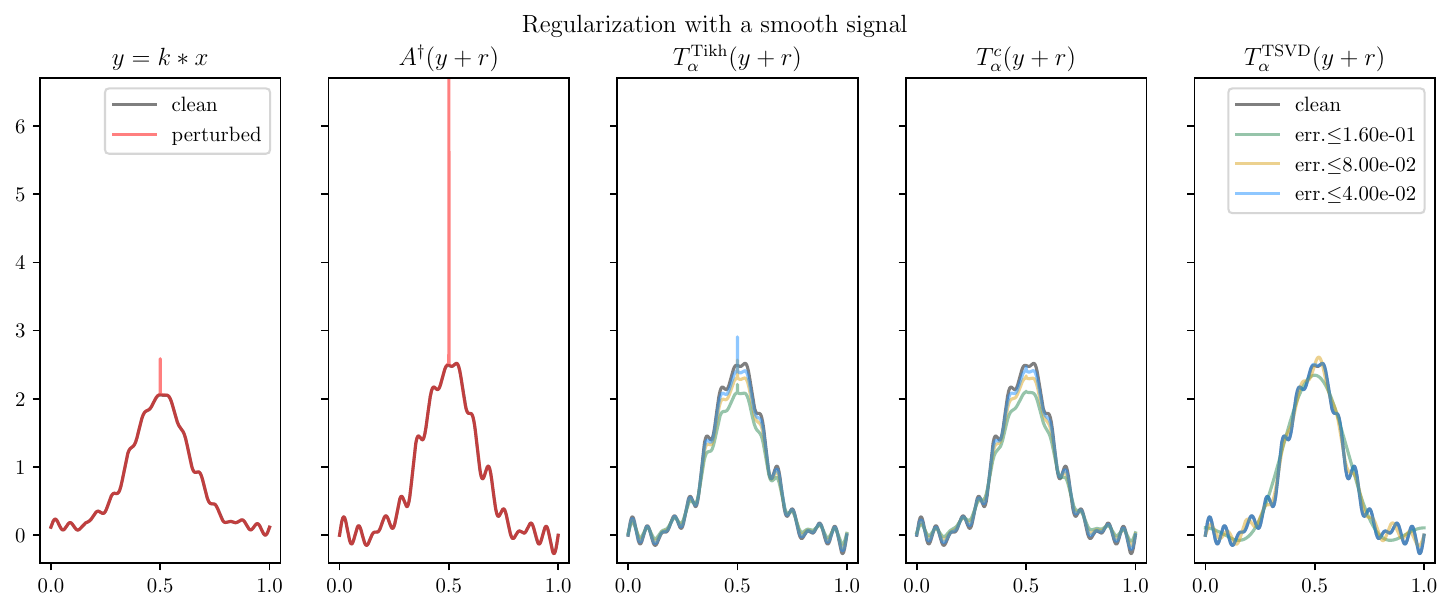}
    \caption{
    Regularization of adversarially perturbed measurements \(y + r\),
    with \(\norm{r}_{L^2(\T)} = 0.005\cdot\norm{y}_{L^2(\T)}\).
    The measurements are \(y=k*x\), where \(x\colon \T\to\R\) is an oscillatory smooth function.
    For each regularization method, \(T_\alpha=T^{\mathrm{Tikh}}_{\alpha},\ T^{c}_{\alpha},\ T^{\mathrm{TSVD}}_{\alpha}\), the parameter \(\alpha\) is chosen small enough so that the relative \(L^2(\T)\)-error, \(\smallnorm{\pseudoinv{A} y - T_\alpha y}_{L^2(\mathcal{\T})}/\norm{y}_{L^2(\mathcal{\T})}\), is less than \(0.16\), \(0.08\), and \(0.04\).}
    \label{fig:perconv-test0}
\end{figure}

\begin{figure}
    \centering
    \includegraphics[width=\textwidth]{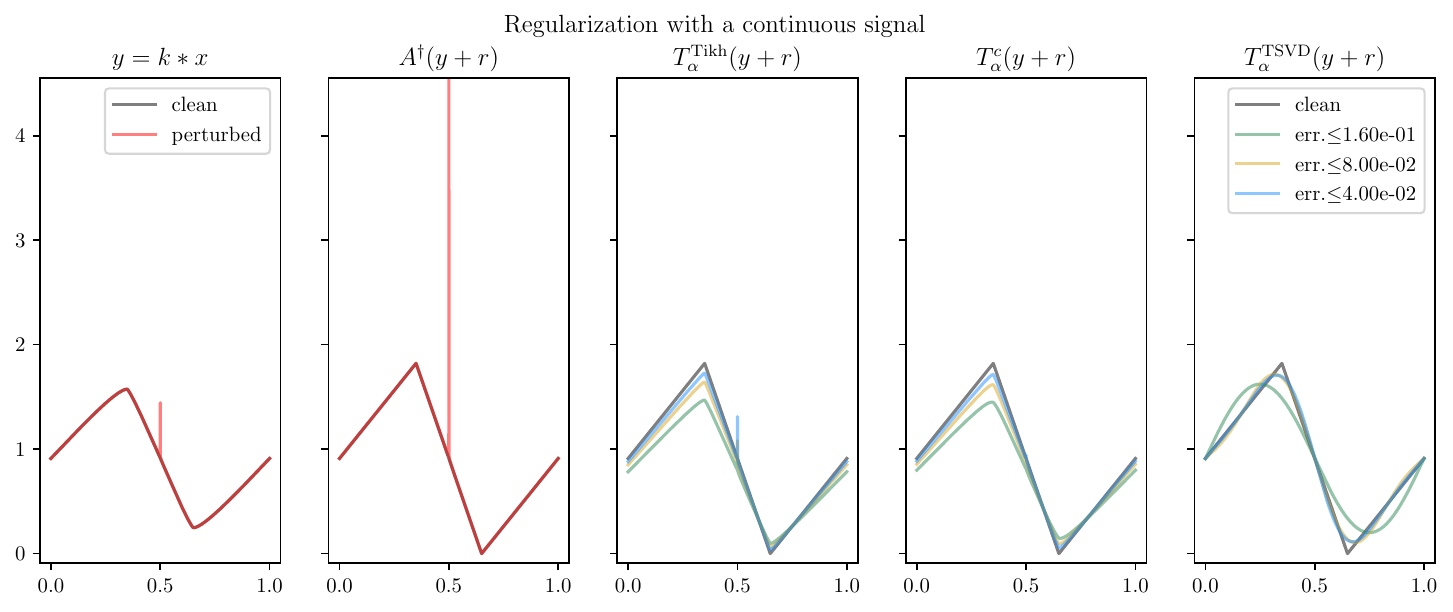}
    \caption{
    Analogous to Fig.~\ref{fig:perconv-test0} but with \(x\colon \T\to\R\) a continuous piece-wise linear function.}
    \label{fig:perconv-test1}
\end{figure}

\begin{figure}
    \centering
    \includegraphics[width=\textwidth]{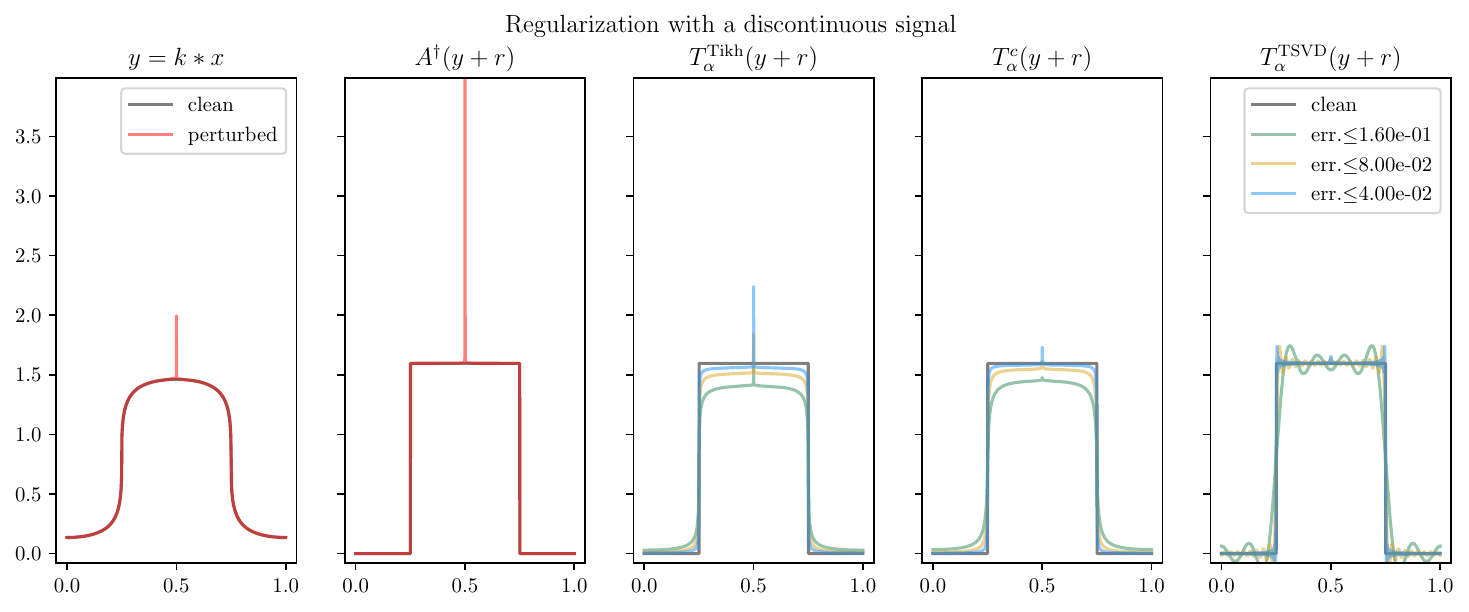}
    \caption{
    Analogous to Fig.~\ref{fig:perconv-test0} but with \(x\colon \T\to\R\) a discontinuous piece-wise constant function.}
    \label{fig:perconv-test2}
\end{figure}

\section*{Acknowledgements}
We are happy to thank Fabio Nicola for useful discussions on specific examples for the wave equation.  Co-funded by the European Union (ERC, SAMPDE, 101041040 and Next Generation EU, Missione 4 Componente 1 CUP D53D23005770006). Views and opinions expressed are however those of the authors only and do not necessarily reflect those of the European Union or the European Research Council. Neither the European Union nor the granting authority can be held responsible for them. This material is based upon work supported by the Air Force Office of Scientific Research under award numbers FA8655-20-1-7027 and FA8655-23-1-7083. The research was supported in part by the MIUR Excellence Department Project awarded to Dipartimento di Matematica, Università di Genova, CUP D33C23001110001.
GSA is a  member of the ``Gruppo Nazionale per l’Analisi Matematica, la Probabilità e le loro Applicazioni'', of the ``Istituto Nazionale di Alta Matematica''. 

\bibliography{ref.bib}

\appendix

\section{Proofs}\label{sec:appendix_proofs}

We collect here some proofs and auxiliary results.\bigskip

\noindent\textbf{Proof of Theorem \ref{thm:advseqadvpert}.}
    We assume that \(A(X\cap X')\subseteq Y\cap Y'\)
    (i.e.\ no adversarial perturbation exists)
    and show that no adversarial sequence can exist.
    
    The intersections \(U = X\cap X'\) and \(V = Y\cap Y'\) are Banach spaces when equipped with the norms
    \begin{equation*}
      \norm{\cdot}_{U}= \max\{ \norm{\cdot}_X,\norm{\cdot}_{X'}\} 
      \quad\text{and}\quad
      \norm{\cdot}_{V}= \max\{ \norm{\cdot}_Y,\norm{\cdot}_{Y'}\}.
    \end{equation*}
    We show that the operator \(A\) is bounded with respect to these norms by using the Closed Graph Theorem.
    Let \(((x_n,y_n))_{n\in\N}\) be a sequence in the graph of \(\restrict{A}{U}\),
    \begin{equation*}
        \Gamma = \left\{(x,y)\in U\times V\colon y=Ax\right\},
    \end{equation*}
    that converges to a point \((u,v)\in U\times V\) in the product norm
    \begin{equation*}
        \norm{(x,y)}_{U\times V}
        = \norm{x}_U + \norm{y}_V,
        \quad
        \forall x\in X\ \forall y\in Y.
    \end{equation*}
    Since
    \begin{equation*}
        \norm{x_n - u}_X
        \leq
        \norm{x_n - u}_U
        \leq\norm{(x_n,y_n) - (u,v)}_{U\times V}
        \xrightarrow[n \to \infty]{} 0,
    \end{equation*}
    we have
    \begin{equation*}
        \norm{y_n - Au}_Y =\norm{Ax_n-Au}_Y \leq \norm{A}_{X\to Y} \norm{x_n - u}_X \xrightarrow[n \to \infty]{} 0.
    \end{equation*}
    Since
    \begin{equation*}
        \norm{y_n - v}_Y \leq \norm{y_n-v}_V \leq \norm{(x_n,y_n) - (u,v)}_{U\times V}
        \xrightarrow[n \to \infty]{} 0,
    \end{equation*}
    we can see that
    the sequence \((y_n)_{n\in\N}\) converges to both \(Au\) and \(v\) in \(Y\).
    This means that \(v=Au\), and hence \((u,v)\in\Gamma\).
    We have shown that \(\Gamma\) is a closed subspace of \(U\times V\),
    and thus that \(A\) is bounded with respect to \(\norm{\cdot}_U\) and \(\norm{\cdot}_V\).
    
    The result now follows.
    Since
    \begin{equation*}
        \norm{Ax}_{Y'} \leq \norm{Ax}_V
        \leq \norm{A}_{U\to V} \norm{x}_U
    \end{equation*}
    for all \(x\in X\),
    any sequence in \(U=X\cap X'\) that is bounded in both \(X\) and \(X'\) has a bounded image in \(Y'\) and hence is  not an adversarial sequence.
\qed\bigskip

The following is an adaptation of Theorem 4.2 from \cite{engl2000regularization} for the specific case of Tikhonov regularization (for which the constant $C$ therein can be chosen equal to $1$).

\begin{lemma}
    \label{thm:tikhonovbounded}
    For all \(y\in Y\), we have
    \begin{equation*}
    \smallnorm{T^{\mathrm{Tikh}}_{\alpha}y}_X
    \leq
    \frac{1}{\sqrt{\alpha}}\smallnorm{y}_Y.
    \end{equation*}
\end{lemma}

\begin{proof}
The proof of Theorem 4.2 in \cite{engl2000regularization} establishes that 
$$
\left\| A A^* (A A^* + \alpha I)^{-1} \right\|_{Y \to Y} \leq 1
$$
and 
$$
\left\|(A A^* + \alpha I)^{-1} \right\|_{Y \to Y} \leq \frac{1}{\alpha}.
$$
Noting that $T_\alpha^{\mathrm{Tikh}} = (A^* A + \alpha I)^{-1} A^*$ and employing the fact
$$
\left(A^* A + \alpha I\right)^{-1} A^* = A^* \left(A A^* + \alpha I \right)^{-1},
$$ we obtain
\begin{align*}
\left\|T_\alpha^{\mathrm{Tikh}} y \right\|^2_X &= \langle A^* \left( A A^* + \alpha I \right)^{-1} y, A^* \left(A A^* + \alpha I \right)^{-1} y\rangle_X\\
& = \langle A A^* \left( A A^* + \alpha I \right)^{-1} y, \left( A A^* + \alpha I \right)^{-1} y \rangle_Y\\
 &\leq \left\| A A^* (A A^* + \alpha I)^{-1} \right\|_{Y \to Y} \left\| (A A^* + \alpha I)^{-1} \right\|_{Y \to Y} \left\|y\right\|_Y^2\\
 &\leq \frac{1}{\alpha} \left\|y\right\|_Y^2.
\end{align*}
\end{proof}
\noindent\textbf{Proof of Lemma \ref{thm:spectralfilter_ratebound}.}
    We derive:
    \begin{align*}
        h'(\sigma) &= \frac{(\eta-4) \, \sigma^{\eta+1} + \alpha \, \eta \, \sigma^{\eta-1}}{(\sigma^2+\alpha)^3}.
    \end{align*}
    This is strictly positive for \(\eta \geq 4\). For $\eta<4$, 
    $h'(\sigma)=0$ if and only if
    \begin{equation*}
        \sigma^2 = \frac{\eta\alpha}{4-\eta},
    \end{equation*}
    which is a maximum (by straightforward calculation, $h''(\sqrt{\frac{\eta \alpha}{4-\eta}})<0$). 
    We evaluate:
    \begin{align*}
        h\left(\sqrt{\frac{\eta\alpha}{4-\eta}}\right)
        &=
        \left(\frac{\eta\alpha}{4-\eta}\right)^{\eta/2}
        \frac{1}{\left(\frac{\eta\alpha}{4-\eta} + \alpha\right)^2}
        \\
        &=
        \alpha^{\eta/2}
        \left(\frac{\eta}{4-\eta}\right)^{\eta/2}
        \frac{1}{\alpha^2\left(\frac{\eta}{4-\eta} + 1\right)^2}
        \\
        &=
        C_\eta
        \alpha^{\eta/2-2},
    \end{align*}
    where $C_\eta = \left(\frac{\eta}{4-\eta}\right)^{\eta/2}  \left(\frac{\eta}{4-\eta} + 1\right)^{-2}.$
\qed
\end{document}